\documentclass{amsart}
\usepackage{amssymb,amsmath,amsthm, amscd, enumerate, mathrsfs}
\usepackage{graphicx}
\usepackage[all]{xy}
\usepackage{tikz}
\usetikzlibrary{intersections, calc}
\usepackage{multicol}

\usepackage{hyperref}

\title[The sarkisov program on log surfaces]{The sarkisov program on log surfaces}
\author{Keisuke Miyamoto}
\date{\today, version 0.27}
\subjclass[2010]{Primary 14E30; Secondary 14E07.}
\keywords{log surfaces, log canonical surfaces, 
Sarkisov program, minimal model program, abundance theorem}
\address{Department of Mathematics, Graduate School of Science, 
Osaka University, Toyonaka, Osaka 560-0043, Japan}
\email{u901548b@ecs.osaka-u.ac.jp}

\DeclareMathOperator{\Supp}{Supp}
\DeclareMathOperator{\Exc}{Exc}
\DeclareMathOperator{\Pic}{Pic}
\DeclareMathOperator{\WDiv}{WDiv}
\DeclareMathOperator{\NS}{NS}
\DeclareMathOperator{\NE}{\overline{NE}}
\DeclareMathOperator{\Char}{char}

\newtheorem{thm}{Theorem}[section]
\newtheorem{lem}[thm]{Lemma}
\newtheorem{prop}[thm]{Proposition}

\newtheorem{cor}[thm]{Corollary}

\theoremstyle{definition}
\newtheorem{ex}[thm]{Example}
\newtheorem{defn}[thm]{Definition}
\newtheorem{rem}[thm]{Remark}
\newtheorem*{ack}{Acknowledgments}

\makeatletter
    
    \@addtoreset{equation}{section}
\makeatother

\begin{document}

\begin{abstract}
We show that the Sarkisov program holds for $\mathbb{Q}$-factorial log surfaces 
and log canonical surfaces over any algebraically closed field. 
\end{abstract}

\maketitle 
\tableofcontents


\section{Introduction}\label{m-sec1}

A log surface is a pair $(X, \Delta)$ consisting of a normal surface $X$ 
and an effective $\mathbb{R}$-divisor $\Delta$ on $X$, whose coefficients are less than or equal to one, such that $K_X+\Delta$ is $\mathbb{R}$-Cartier. 
In addition, we say that $(X, \Delta)$ is a $\mathbb{Q}$-factorial log surface if $X$ is $\mathbb{Q}$-factorial. 
Then we emphasize that $(X, \Delta)$ is not necessarily log canonical. 
Fujino and Tanaka established the minimal model theory for $\mathbb{Q}$-factorial log surfaces 
and log canonical surfaces in full generality (\cite{fujino-surfaces} and \cite{tanaka}). 
We also note that a log canonical surface is not necessarily $\mathbb{Q}$-factorial. 
Therefore, if $(X, \Delta)$ is a $\mathbb{Q}$-fatorial log surface (resp.~log canonical surface), 
then we obtain a minimal model $(X^+, \Delta^+)$ of $(X, \Delta)$, 
which is uniquely determined by $(X, \Delta)$ (see \cite[Proposition 3.9]{fujino-surfaces}), 
or a Mori fiber space $\phi:(X^+, \Delta^+)\to S$ 
by running the minimal model program for $(X, \Delta)$. 
The Mori fiber space $\phi:(X^+, \Delta^+)\to S$ is not uniquely determined by $(X, \Delta)$. 
It depends on how to contract negative curves.

In this paper, we prove the following two theorems:

\begin{thm}[Main Theorem I]\label{main-theorem1}
Let $(Z, \Phi)$ be a projective $\mathbb{Q}$-factorial log surface. 
Let $\phi:X\to S$ and $\psi:Y\to T$ be two Mori fiber spaces, 
which are outputs of the $(K_Z+\Phi)$-minimal model program.

Then the induced birational map $\sigma:X\dashrightarrow Y$ 
is a composition of Sarkisov links. 
\end{thm}

\begin{thm}[Main Theorem II]\label{main-theorem2}
Let $(Z, \Phi)$ be a projective log canonical surface. 
Let $\phi:X\to S$ and $\psi:Y\to T$ be two Mori fiber spaces, 
which are outputs of the $(K_Z+\Phi)$-minimal model program.

Then the induced birational map $\sigma:X\dashrightarrow Y$ 
is a composition of Sarkisov links. 
\end{thm}

Next we quickly recall the definition of Sarkisov links. 

\begin{defn}[Sarkisov links]
Let $(Z, \Phi)$ be a projective $\mathbb{Q}$-factorial log surface (or log canonical surface). 
Let $\phi:X\to S$ and $\psi:Y\to T$ be two Mori fiber spaces, 
which are outputs of the $(K_Z+\Phi)$-minimal model program.

A Sarkisov link $\sigma:X\dashrightarrow Y$ between $\phi$ and $\psi$ 
is one of four types: 

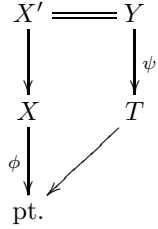
\begin{multicols}{2}
Type (I)
$$
\xymatrix{
X'\ar@{=}[r]\ar[d]&Y\ar[d]^\psi \\
X\ar[d]_\phi&T\ar[ld] \\
\text{pt.}
}
$$
Type (III)
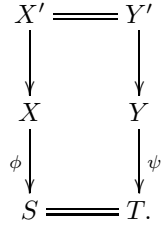
$$
\xymatrix{
X\ar@{=}[r]\ar[d]_\phi&Y'\ar[d] \\
S\ar[rd]&Y\ar[d]^\psi \\
&\text{pt.}
}
$$
Type (II)
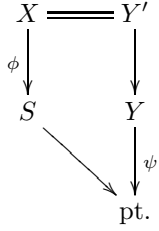
$$
\xymatrix{
X'\ar@{=}[r]\ar[d]&Y'\ar[d] \\
X\ar[d]_\phi&Y\ar[d]^\psi \\
S\ar@{=}[r]&T. 
}
$$
Type (IV)
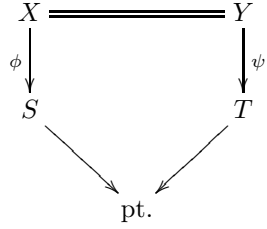
$$
\xymatrix{ 
X\ar@{=}[rr]\ar[d]_\phi&&Y\ar[d]^\psi \\
S\ar[rd]&&T\ar[ld] \\
&\text{pt.}
}
$$
\end{multicols}

Every vertical arrow is an extremal contraction. 
Moreover if the target is $X$ or $Y$, then it is a divisorial contraction. 
The space $X'$ and $Y'$ are realized as the ample model of $(Z, \Psi)$ 
for some boundary $\mathbb{R}$-divisor $\Psi$. 
\end{defn}

The Sarkisov links for smooth surfaces are well-known (see \cite{matsuki}): 
A link of Type (II) is an elementary transformation of $\mathbb{P}^1$-bundles over curves. 
In a link of Type (I), $X'\to X$ is nothing but a blow-up of a smooth point of $\mathbb{P}^2$. 
A link of Type (III) is the same diagram, reflected in a vertical line, as a link of Type (I). 
In a link of Type (IV), $X=Y=\mathbb{P}^1\times\mathbb{P}^1$, and $\phi:X\to S$ and $\psi:Y\to T$ 
are projections to $\mathbb{P}^1$.

For singular surfaces, there are many Sarkisov links. 
Let us see some examples of Sarkisov links for $\mathbb{Q}$-factorial surfaces. 

\begin{ex}\label{ex1}
We fix a lattice $N=\mathbb{Z}^2$ and take lattice points 
\[v_1=(1, 0), v_2=(1, n), v_3=(-1, 0), v_4=(-1, -n)\]
where $n$ is an integer with $n\geq2$.

We consider the following fan 
\begin{align*}
\Delta=\{&\mathbb{R}_{\geq0}v_1+\mathbb{R}_{\geq0}v_2, 
\mathbb{R}_{\geq0}v_2+\mathbb{R}_{\geq0}v_3, 
\mathbb{R}_{\geq0}v_3+\mathbb{R}_{\geq0}v_4, \\
&\mathbb{R}_{\geq0}v_4+\mathbb{R}_{\geq0}v_1 
\text{ and their faces}\}.
\end{align*}
Then the associated toric variety $X=X(\Delta)$ is a projective $\mathbb{Q}$-factorial toric surface. 
We consider $\mathbb{R}^2\to\mathbb{R}$ defined by $(x, y)\mapsto y$ 
and $\mathbb{R}^2\to\mathbb{R}$ defined by $(x, y)\mapsto nx-y$.

Then we obtain two different toric morphisms $p_1;p_2:X\to \mathbb{P}^1$. 
This gives a Sarkisov link of Type (IV). 
$$
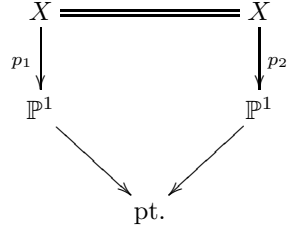
\xymatrix{
X\ar@{=}[rr]\ar[d]_{p_1}&&X\ar[d]^{p_2} \\
\mathbb{P}^1\ar[rd]&&\mathbb{P}^1\ar[ld]\\
&\text{pt.}
}
$$
We note that $X\not\simeq\mathbb{P}^1\times\mathbb{P}^1$. 
\end{ex}

\begin{ex}
In Example \ref{ex1}, we put 
\[v_1=(1, 0), v_2=(1, n), v_3=(0, 1), v_4=(-1, -1)\]
where $n$ is an integer with $n\geq2$.\\  
Then the associated toric variety $X=X(\Delta)$ 
is a projective $\mathbb{Q}$-factorial toric surface with $\rho(X)=2$. 
By removing $v_2$, we obtain a toric birational morphism $X\to\mathbb{P}^2$. 
By removing $v_3$, we get $X\to\mathbb{P}(1, n-1, n)$. 
Thus we have a Sarkisov link of Type (II). 
$$
\xymatrix{ 
X\ar@{=}[r]\ar[d]&Y\ar[d] \\
\mathbb{P}^2\ar[d]&\mathbb{P}(1, n-1, n)\ar[d] \\
\text{pt}\ar@{=}[r]&\text{pt.}
}
$$
\end{ex}

\begin{ex}
In Example \ref{ex1}, we put 
\[v_1=(1, 0), v_2=(1, n), v_3=(0, 1), v_4=(-1, -n)\]
where $n$ is an integer with $n\geq2$.\\ 
We consider the associated toric variety $X=X(\Delta)$. 
By removing $v_2$, we have a birational contraction morphism $X\to\mathbb{P}(1, 1, n)$. 
By considering $\mathbb{R}^2\to\mathbb{R}$ defined by $(x, y)\mapsto nx-y$, 
we can construct $X\to\mathbb{P}^1$. 
Thus we obtain the following commutative diagram: 
$$
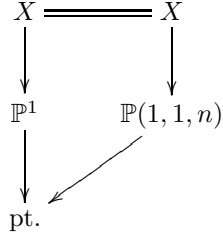
\xymatrix{
X\ar@{=}[r]\ar[d]&X\ar[d] \\
\mathbb{P}^1\ar[d]&\mathbb{P}(1, 1, n)\ar[ld] \\
\text{pt.}
}
$$
So we have Sarkisov links of Type (I) and (III). 
\end{ex}

The Sarkisov program for smooth surafaces was generalized 
for 3-folds with some singularities by Corti following Sarkisov's idea (see \cite{corti} and \cite{bm}). 
Moreover Hacon-M\textsuperscript{c}Kernan generalized the Sarkisov program 
for any dimensional kawamata log terminal pairs in full generality (\cite{hacon-mackernan}). 
We emphasize that these arguments 
are carried out over the complex number field $\mathbb{C}$. 
However, Hacon-M\textsuperscript{c}Kernan's proof is quite different from the original one 
used by Corti or Bruno-Matsuki. 
Roughly speaking, in the original proof, we keep track of tree invariants, 
called the Sarkisov degree, associated with the singularities 
and we have to show the Sarkisov degree satisfies the ascending chain condition. 
This proof heavily depends on the singularities and so we can not use this idea. 
In fact, we do not assume a $\mathbb{Q}$-factorial log surface is even log canonical. 
By contrast, in their proof, we use ``finiteness of ample models''
(Theorem \ref{finiteness-of-models} and Theorem \ref{finiteness-of-models+}) 
instead of the Sarkisov degree. 
By this property we see the correspondence of 
considering some polytopes consisted of the divisors with the same ample model 
to Sarkisov links (Theorem \ref{hm3.7} and Theorem \ref{hm3.7+}). 
This proof does not depend on the singularities and so we prove the Sarkisov program by using their idea. 
We note that this idea is based on Shokurov's one called the 
geography of log models.

Finally we summarize the contents of this paper. 
In Section 2, we quickly recall some basic definitions for log surfaces. 
In Section 3, we collect some results of log surfaces without proof. 
In Section 4, we define ``special'' (Definition \ref{special}). 
This concept plays a crucial role in subsequent sections. 
Thanks to this idea, we can treat $\mathbb{Q}$-factorial log surfaces 
and log canonical surfaces simultaneously. 
Moreover we prove some results of \cite{bchm} for log surfaces. 
In Section 5, we analyze the wall-crossing phenomena of ample models. 
In Section 6, we complete the proof of Main Theorems. 
In Section 7, we prove that certain spaces appeared in Sarkisov links are 
isomorphic to $\mathbb{P}^1$.

Throughout this paper, we will work over an algebraically closed field of any characteristic. 

\begin{ack}
The author would like to thank Professor Osamu Fujino 
for various suggestions and warm encouragement. 
He also thanks his colleagues for discussions. 
Finally, he thanks Professor Vyacheslav Vladimirovich Shokurov 
for comments. 
\end{ack}


\section{Preliminaries}\label{m-sec2} 

We recall some basic definitions. 

\begin{defn}[$\mathbb{Q}$-divisors and $\mathbb{R}$-divisors]
Let $\pi:X\to U$ be a proper morphism between normal quasi-projective varieties.

Two $\mathbb{R}$-divisors $D$ and $D'$ on $X$ are $\mathbb{R}$-linearly equivalent over $U$ 
(or~$D\sim_{\mathbb{R}, U}D'$) if there is an $\mathbb{R}$-Cartier divisor $B$ on $U$ 
such that $D-D'\sim_\mathbb{R}\pi^*B$.

Two $\mathbb{R}$-divisors $D$ and $D'$ are numerically equivalent over $U$
(or~$D\equiv_UD'$) if $D-D'$ is an $\mathbb{R}$-Cartier divisor 
and $(D-D')\cdot C=0$ for any curve $C\subset X$ contained in a fiber of $\pi$.

An $\mathbb{R}$-Cartier divisor $D$ on $X$ is nef over $U$ (or~$\pi$-nef) 
if $D\cdot C\geq 0$ for any curve $C\subset X$ contracted by $\pi$.

An $\mathbb{R}$-Cartier divisor $D$ is semiample over $U$ (or~$\pi$-semiample)
if there is a morphism $f:X\to Y$ over $U$ such that $D\sim_{\mathbb{R}, U}f^*H$,
where $H$ is ample over $U$.

An $\mathbb{R}$-divisor $D$ is big over $U$ (or~$\pi$-big) 
if there are an ample $\mathbb{R}$-divisor $A$ and 
an effective $\mathbb{R}$-divisor $B\geq 0$ such that $D\sim_{\mathbb{R}, U}A+B$.

An $\mathbb{R}$-divisor $D$ is pseudo-effective over $U$ (or~$\pi$-pseudo-effective) 
if its numerical class belongs to the closure of the cone of big divisors over $U$.

We frequently write $D\in[0, 1]$ (resp.~$D\in[0, 1)$) 
if an $\mathbb{R}$-divisor $D$ is effective 
and its coefficients are less than or equal to one 
(resp.~less than one). 
An $\mathbb{R}$-divisor $D$ is said to be a boundary $\mathbb{R}$-divisor 
if $D\in[0, 1]$. 
$\mathbb{Q}$-divisors are defined similarly. 
\end{defn}

\begin{defn}[Singularities of pairs]
Let $X$ be a normal variety and let $\Delta$ be an effective $\mathbb{R}$-divisor on $X$ 
such that $K_X+\Delta$ is $\mathbb{R}$-Cartier. 
Let $f:Y\to X$ be a log resolution of the pair $(X, \Delta)$. 
We write 
\[K_Y=f^*(K_X+\Delta)+\sum a_iE_i.\]
We say that $(X, \Delta)$ is kawamata log terminal (resp.~log canonical) 
if $a_i>-1$ (resp.~$a_i\geq-1$) for every $i$. 
We say that $X$ is $\mathbb{Q}$-factorial if every Weil divisor on $X$ is $\mathbb{Q}$-Cartier. 
\end{defn}

\begin{defn}[Linear systems]
Let $\pi:X\to U$ be a projective morphism between normal varieties. 
Let $D$ be an $\mathbb{R}$-divisor on $X$. The real linear system associated to $D$ over $U$ is 
\[|D/U|_\mathbb{R}=\{B\geq0\mid B\sim_{\mathbb{R}, U}D\}.\]
The stable fixed divisor is the divisorial support 
of the intersection of all components of 
the real linear system $|D/U|_\mathbb{R}$. 
\end{defn}

We introduce some definitions and notations of \cite{bchm} for surfaces. 

\begin{defn}
Let $\pi:X\to U$ be a projective morphism from a normal surface to a normal quasi-projective variety. 
Let $f:X\to Y$ be a birational morphism between normal surfaces over $U$ 
and let $D$ be an $\mathbb{R}$-Cartier divisor on $X$ such that $f_*D$ is also $\mathbb{R}$-Cartier.

We say that $f$ is $D$-non-positive (resp.~$D$-negative) 
if $E=D-f^*f_*D$ is an effective $f$-exceptional divisor 
(resp.~an effective $f$-exceptional divisor and the support of $E$ contains all $f$-exceptional divisors).

We say that $g:X \to Z$ is the ample model of $D$ over $U$ 
if $g$ is a contraction morphism, $Z$ is normal and projective over $U$ 
and there is an ample divisor $H$ over $U$ on $Z$ such that 
$D\sim_{\mathbb{R}, U}g^*H+E$ 
with $B\geq E$ for every $B\in|D/U|_\mathbb{R}$, 
where $E\geq 0$. 
We note that $g$ is not necessarily birational. 
\end{defn}

\begin{defn}
Let $\pi:X\to U$ be a projective morphism from a normal surface to a normal quasi-projective variety. 
Let $(X, \Delta)$ be a $\mathbb{Q}$-factoraial log surface (or~log canonical surface). 
Let $f:X\to Y$ be a birational contraction morphism over $U$ between normal surfaces 
and let $g:X\to S$ be a contraction morphism over $U$ to a normal quasi-projective variety.

We say that $f$ is a weak log canonical model of $K_X+\Delta$ over $U$ 
if $f$ is $(K_X+\Delta)$-non-positive and $K_Y+f_*\Delta$ is nef over $U$.

We say that $f$ is the log canonical model of $K_X+\Delta$ over $U$ 
if $f$ is the ample model of $K_X+\Delta$ over $U$.

We say that $f$ is a minimal model of $K_X+\Delta$ over $U$ 
if $f$ is $(K_X+\Delta)$-negative and $K_Y+f_*\Delta$ is nef over $U$.

We say that $g$ is a Mori fiber space of $K_X+\Delta$ over $U$
if $-(K_X+\Delta)$ is $g$-ample, $\rho(X/S)=1$ and $\dim S<2$. 
It is obvious that if $S$ is a point (resp.~curve), then $\rho(X)=1$ (resp.~$\rho(X)=2$).

We say that $f$ is the result of running the $(K_X+\Delta)$-minimal model program over $U$ 
if $f$ is any sequence of divisorial contractions for the $(K_X+\Delta)$-minimal model program over $U$. 
We emphasize that the result of running the $(K_X+\Delta)$-minimal model program over $U$
is not necessarily a minimal model of $K_X+\Delta$ over $U$ 
or a Mori fiber space of $K_X+\Delta$ over $U$. 
\end{defn} 

\begin{defn}
Suppose that $\pi:X\to U$ is a projective morphism from a normal surface
to a normal quasi-projective variety, $V$ is a finite dimensional affine subspace 
of the real vector space $\WDiv_\mathbb{R}(X)=\WDiv(X)\otimes_\mathbb{Z}\mathbb{R}$, 
which is defined over the rationals, where $\WDiv(X)$ is the group of Weil divisors on $X$, 
and $A\geq 0$ is a $\mathbb{Q}$-divisor on $X$. 
Let $\mathcal{P}$ be a rational polytope in $\WDiv_\mathbb{R}(X)$. 
Then we define
\begin{align*}
V_A&=\{\Delta=A+B\mid B\in V\}, \\
\mathcal{B}_A(V)&=\{\Delta=A+B\in V_A\mid\Delta\in[0, 1] \text{ and } B\geq 0\}, \\
\mathcal{L}_A(V)&=\{\Delta=A+B\in V_A\mid K_X+\Delta \text{ is log canonical and }B\geq0\}, \\
\mathcal{E}_{A, \pi}(V)&=\{\Delta\in\mathcal{P}\mid K_X+\Delta\text{ is pseudo-effective over }U\}, \\
\mathcal{N}_{A, \pi}(V)&=\{\Delta\in\mathcal{P}\mid K_X+\Delta\text{ is nef over }U\}. 
\end{align*}
We can easily check that $\mathcal{B}_A(V)$ is a rational polytope 
and it is well-known that so is $\mathcal{L}_A(V)$. 
If we consider $\mathbb{Q}$-factorial log surfaces (resp.~log canonical surfaces), 
then $\mathcal{P}=\mathcal{B}_A(V)$ (resp, $\mathcal{P}=\mathcal{L}_A(V))$. 
Let $f:X\to Y$ be a birational contraction over $U$ and let $g:X\to Z$ be a contraction morphism over $U$. 
Then we define
\begin{align*}
\mathcal{W}_{f, A, \pi}(V)&=\{\Delta\in\mathcal{E}_{A, \pi}(V)\mid f 
\text{ is a weak log canonical model of } K_X+\Delta\text{ over }U\}, \\
\mathcal{A}_{g, A, \pi}(V)&=\{\Delta\in\mathcal{E}_{A, \pi}(V)\mid g 
\text{ is the ample model of }K_X+\Delta\text{ over }U\}. 
\end{align*}
In addition, let $\mathcal{C}_{g, A, \pi}(V)$ denote the closure of $\mathcal{A}_{g, A, \pi}(V)$. 
We simply write $\mathcal{E}_{A}(V)=\mathcal{E}_{A, \pi}(V)$ 
(resp.~$\mathcal{N}_{A}$, $\mathcal{W}_{f, A}$ and $\mathcal{A}_{g, A}$) when $U$ is a point. 
\end{defn}


\section{Known results of log surfaces}\label{m-sec3}

In this section, we collect some known results of $\mathbb{Q}$-factorial log surfaces 
and log canonical surfaces without proof. 
These results are established by Fujino in characteristic zero 
and generalized by Tanaka in positive characteristic. 

\begin{thm}[Minimal model program for log surfaces]\label{mmp} 
Let $(X, \Delta)$ be a log surface 
and let $\pi:X\to U$ be a projective morphism to a variety. 
Assume that one of the following conditions holds: 
\begin{itemize}
\item[(1)] $X$ is $\mathbb{Q}$-factorial, or
\item[(2)] $(X, \Delta)$ is log canonical. 
\end{itemize}
Then we may run the $(K_X+\Delta)$-minimal model program over $U$. 
In other words, there is a sequence of divisorial contractions
$$
\xymatrix{
(X, \Delta)=(X_0, \Delta_0)\ar[r]^-{\phi_0}  & (X_1, \Delta_1) \ar[r]^-{\phi_1} 
& \cdots \ar[r]^-{\phi_{n-1}} & (X_n, \Delta_n)=(X^+, \Delta^+)
}
$$
such that $(X^+, \Delta^+)$ is a minimal model over $U$ 
or $(X^+, \Delta^+)\to S$ is a Mori fiber space over $U$. 
\end{thm}
\begin{proof}
See \cite[Theorem 3.1] {fujino-surfaces} and \cite[Theorem 6.5]{tanaka}. 
\end{proof}

\begin{rem}\label{output}
We note that $X_i$ is $\mathbb{Q}$-factorial (resp.~$(X_i, \Delta_i)$ is log canonical) 
for every $i$ in Case (1) (resp.~(2)). 
\end{rem}

\begin{thm}[Abundance theorem] 
Let $(X, \Delta)$ be a log surface 
and let $\pi:X\to U$ be a proper surjective morphism to a variety. 
Assume that one of the following conditions holds: 
\begin{itemize}
\item[(1)] $X$ is $\mathbb{Q}$-factorial, or
\item[(2)] $(X, \Delta)$ is log canonical. 
\end{itemize}
If $K_X+\Delta$ is nef over $U$, then $K_X+\Delta$ is semiample over $U$. 
\end{thm}
\begin{proof}
See \cite[Theorem 8.1]{fujino-surfaces} and \cite[Corollary 6.10]{tanaka}. 
\end{proof}

\begin{prop}[Extremal rational curves] 
Let $(X, \Delta)$ be a log surface 
and let $\pi:X\to U$ be a projective surjective morphism to a variety. 
If $R$ is a $(K_X+\Delta)$-negative extremal ray, 
then there is a rational curve $C$ on $X$ such that $C$ generates $R$ 
and $-(K_X+\Delta)\cdot C\leq 3$. 
\end{prop}
\begin{proof}
See \cite[Proposition 3.8]{fujino-surfaces} and \cite[Proposition 3.39]{tanaka}. 
\end{proof}

The following lemmas are proved by using the proposition above. 
For details, see \cite[Proposition 3.2]{birkar}. 

\begin{lem}\label{polytope1}
Let $\pi:X\to U$ be a proper surjective morphism from a $\mathbb{Q}$-factorial surface to a variety. 
Fix a $\mathbb{Q}$-divisor $C\geq0$.
Let $V$ be a finite dimensional affine subspace of $\WDiv_\mathbb{R}(X)$, 
which is defined over the rationals. 
Then 
\[\mathcal{N}_{C, \pi}(V)=\{\Delta\in\mathcal{B}_C(V)\mid K_X+\Delta\text{ is nef over }U\}\]
is a rational polytope. 
\end{lem}

\begin{lem}\label{polytope2}
Let $\pi:X\to U$ be a proper surjective morphism from a normal surface to a variety. 
Fix a $\mathbb{Q}$-divisor $C\geq0$.
Let $V$ be a finite dimensional affine subspace of $\WDiv_\mathbb{R}(X)$, 
which is defined over the rationals. 
Then 
\[\mathcal{N}_{C, \pi}(V)=\{\Delta\in\mathcal{L}_C(V)\mid K_X+\Delta\text{ is nef over }U\}\]
is a rational polytope. 
\end{lem}

We close this section with the following theorem. 
However we use this theorem only in Section 7. 

\begin{thm}\label{oo}
Let $(X, \Delta)$ be a projective log surface 
such that $\lfloor \Delta\rfloor=0$ and $-(K_X+\Delta)$ is nef and big. 
Then $\Pic(X)$ is a free abelian group of finite rank. 
In particular, there is no non-trivial torsion line bundle on $X$. 
\end{thm}
\begin{proof}
See \cite[Theorem 1.5]{oo}. 
\end{proof}


\section{Some results of \cite{bchm} for log surfaces}\label{m-sec4}

We start with a crucial definition 
to treat $\mathbb{Q}$-factorial surfaces and log canonical surfaces simultaneously. 

\begin{defn}[special]\label{special}
Let $\pi:X\to U$ be a projective morphism from a normal surface to a normal quasi-projective variety. 
Let $f:X\to Y$ be a birational contraction between normal surfaces over $U$.

Then we say that $f$ is special if there are two Zariski open sets 
$U_X\subset X$ and $U_Y\subset Y$ such that $f$ induces $U_X\simeq U_Y$ 
and, $U_X$ and $U_Y$ contain all non-$\mathbb{Q}$-factorial singularities respectively. 
In particular, the exceptional locus $\Exc(f)$ does not pass through 
any non-$\mathbb{Q}$-factorial singularities. 
\end{defn}

We can show the Sarkisov program on log canonical surfaces 
similarly to on $\mathbb{Q}$-factorial surfaces thanks to the following theorem and easy lemma. 

\begin{thm}\label{special1}
Let $f:X\to Y$ be a projective birational morphism between normal surfaces. 
Let $(X, \Delta)$ be log canonical. 
If $-(K_X+\Delta)$ is $f$-ample, 
then every $f$-exceptional curve is $\mathbb{R}$-Cartier. 
In particular, all divisorial contractions are special. 
\end{thm}
\begin{proof}
See \cite[Theorem 4.1]{fujino-sing}. 
\end{proof}

\begin{lem}\label{special2}
Let $\pi:X\to U$ be a projective morphism from a normal surface to a normal quasi-projective variety. 
Let $(X, \Delta)$ be a log canoincal surface
and  let $\phi:X\to Y$ be a minimal model of $K_X+\Delta$ over $U$. 
Then $\phi$ is special. 
\end{lem}
\begin{proof}
If we put $U_X=X\setminus \Exc(\phi)$ and $U_Y=Y\setminus \phi(\Exc(\phi))$, 
the conditions above are satisfied by Theorem \ref{special1}. 
\end{proof}

The following lemma also plays a crucial role in this section. 
 
\begin{lem}\label{q-fac1}
Let $(X, \Delta)$ be a $\mathbb{Q}$-factorial log surface such that $\Delta\in[0, 1)$. 
If $f:X\to Y$ is $(K_X+\Delta)$-non-positive, 
then $Y$ is also $\mathbb{Q}$-factorial. 
\end{lem}
\begin{proof}
We put $E=\Exc(f)$ and fix a sufficiently small number $0<\epsilon<1$ 
such that $\Delta+\epsilon E\in[0, 1)$. 
Then we may run the $(K_X+\Delta+\epsilon E)$-minimal model program $g:X\to Z$ over $Y$. 
By the negativity lemma $E$ is $g$-exceptional. 
Thus $f$ and $g$ contract the same curves 
and so $Y$ and $Z$ are isomorphic. 
This means that $Y$ is $\mathbb{Q}$-factorial (cf.~Remark \ref{output}). 
\end{proof}

\begin{rem}\label{q-fac2}
By the proof of Lemma \ref{q-fac1}, the log canonical model of a $\mathbb{Q}$-factorial log surface 
whose boundary divisor contained in $[0, 1)$
is also $\mathbb{Q}$-factorial. 
\end{rem}

The next lemma is well-known properties of ample models. 
For details, see \cite[Lemma 3.6.6]{bchm}. 

\begin{lem}\label{ample-model}
Let $\pi:X\to U$ be a projective morphism from a normal surface 
to a normal quasi-projective variety 
and let $D$ be an $\mathbb{R}$-Cartier divisor on $X$. 
\begin{itemize}
\item[(1)] If $g_i:X\to X_i$ are two ample models of $D$ over $U$ $(i=1, 2)$, 
then there is an isomorphism $\chi:X_1\to X_2$ such that $g_2=\chi\circ g_1$. 
\item[(2)] If $f:X\to Y$ is a weak log canonical model of $D$ over $U$, 
then there are the ample model $g:X\to Z$ of $D$ over $U$ 
and a contraction morphism $h:Y\to Z$ such that $g=h\circ f$ 
and $f_*D\sim_{\mathbb{R}, U}h^*H$, 
where $H$ is an ample divisor corresponding to the ample model $g$ of $D$ over $U$. 
\end{itemize}
\end{lem}
\begin{proof}
(1) We put $D\sim_{\mathbb{R}, U}g_i^*H_i+E_i$ given by definition of ample models. 
Since there is no stable fixed divisor of $g_1^*H_1$, we obtain $E_1\geq E_2$. 
By symmetry, $E_1=E_2$. 
Thus we obtain $g_1^*H_1\sim_{\mathbb{R}, U}g_2^*H_2$. 
Then $g_1$ and $g_2$ contract the same curves 
and so there is an isomorphism $\chi:X_1\to X_2$.

(2) By the abundance theorem, we obtain a contraction morphism $h:Y\to Z$ over $U$ 
such that $f_*D\sim_{\mathbb{R}, U}h^*H$, where $H$ is ample over $U$. 
If we put $g=h\circ f$ and $E=D-f^*f_*D$, 
then $E$ is effective and $f$-exceptional. 
If $B\in|D/U|_\mathbb{R}$, then $B\geq E$ by the negativity lemma. 
Therefore $g$ is the ample model of $D$ over $U$. 
\end{proof} 

\begin{lem}\label{bchm3.6.12}
Let $\pi:X\to U$ be a projective morphism from a normal surface to a normal quasi-projective variety. 
Let $\phi:X\to Y$ be a birational contraction over $U$. 
Let $A\geq0$ be an ample $\mathbb{Q}$-divisor over $U$. 
Let $V$ be a finite dimensional affine subspace of $\WDiv_\mathbb{R}(X)$ 
such that $\mathcal{B}_A(V)$ spans $\WDiv_\mathbb{R}(X)$ modulo numerical equivalence over $U$ 
and $\mathcal{W}_{\phi, A, \pi}(V)$ intersects the interior of $\mathcal{B}_A(V)$.

If $X$ and $Y$ are both $\mathbb{Q}$-factorial, then 
\[\mathcal{W}_{\phi, A, \pi}(V)=\mathcal{C}_{\phi, A, \pi}(V).\]
\end{lem}

\begin{proof}
We have $\mathcal{W}_{\phi, A, \pi}(V)\supset\mathcal{A}_{\phi, A, \pi}(V)$. 
In fact, if $\Delta\in\mathcal{A}_{\phi, A, \pi}(V)$, 
then there are an ample divisor $H$ over $U$ on $Y$ 
and an effective $\mathbb{R}$-divisor $E$ 
such that $K_X+\Delta\sim_{\mathbb{R}, U}\phi^*H+E$ 
and $B\geq E$ for any $B\in|K_X+\Delta/U|_\mathbb{R}$. 
We may find a $\phi$-exceptional divisor $F\geq0$ such that $\phi^*H-F$ is ample over $U$ 
and if $\epsilon>0$ is sufficiently small, then $\phi^*H-F+\epsilon E$ is ample over $U$. 
We take a component $B$ of $E$. 
Since there is no stable fixed divisor of $\phi^*H-F+\epsilon E$, $B$ is a component of $F$ 
and so $E$ is $\phi$-exceptional. 
Thus $\phi_*(K_X+\Delta)\sim_{\mathbb{R}, U}H$ is ample over $U$ 
and so $\phi$ is a weak log canonical model of $K_X+\Delta$ over $U$. 
Since $\mathcal{W}_{\phi, A, \pi}$ is closed by definition, 
$\mathcal{W}_{\phi, A, \pi}(V)\supset\bar{\mathcal{A}}_{\phi, A, \pi}(V)=\mathcal{C}_{\phi, A, \pi}(V)$.

We show the opposite inclusion. 
Then it is sufficient to prove that a dense subset 
of $\mathcal{W}_{\phi, A, \pi}(V)$ is contained in $\mathcal{A}_{\phi, A, \pi}(V)$.

We take $\Delta$ belonging to the interior of $\mathcal{W}_{\phi, A, \pi}(V)$. 
We take a general ample $\mathbb{Q}$-divisor $H$ over $U$ on $Y$. 
We put $H'=\phi^*H$ and then $\phi$ is $H'$-non-positive. 
We may find $\Delta'\in\mathcal{B}_A(V)$ such that $B=\Delta'-\Delta$ 
is numerically equivalent to $\eta H'$ over $U$ for some $\eta>0$. 
Replacing $H$ by $\eta H$, we may assume that $\eta=1$. 
Then $\phi$ is $(K_X+\Delta+\lambda B)$-non-positive 
and $\phi_*(K_X+\Delta+\lambda B)$ is ample over $U$ 
for any $\lambda>0$. 
Now we have 
\[\Delta+\lambda B=\Delta+\lambda(\Delta'-\Delta)\in\mathcal{B}_A(V)\]
for any $\lambda\in[0, 1]$. 
But then $\phi$ is the ample model of $K_X+\Delta+\lambda B$ over $U$ for any $\lambda\in(0, 1]$. 
Therefore we obtain $\mathcal{W}_{\phi, A, \pi}(V)\subset\mathcal{C}_{\phi, A, \pi}(V)$.
\end{proof}

The following assertion is the log canonical version of Lemma \ref{bchm3.6.12}. 

\begin{lem}\label{bchm3.6.12+}
Let $\pi:X\to U$ be a projective morphism from a normal surface to a normal quasi-projective variety. 
Let $\phi:X\to Y$ be a birational contraction over $U$. 
Let $A\geq0$ be an ample $\mathbb{Q}$-divisor over $U$. 
Let $V$ be a finite dimensional affine subspace of $\WDiv_\mathbb{R}(X)$ 
such that $\mathcal{L}_A(V)$ spans $\WDiv_\mathbb{R}(X)$ modulo numerical equivalence over $U$ 
and $\mathcal{W}_{\phi, A, \pi}(V)$ intersects the interior of $\mathcal{L}_A(V)$.

If $\phi$ is special, then 
\[\mathcal{W}_{\phi, A, \pi}(V)=\mathcal{C}_{\phi, A, \pi}(V).\]
\end{lem}
\begin{proof}
It is clear that $\mathcal{W}_{\phi, A, \pi}(V)\supset\mathcal{C}_{\phi, A, \pi}(V)$. 
In fact, any component of $\Exc(\phi)$ is $\mathbb{R}$-Cartier 
since $\phi$ is special. 
Thus the rest of the proof is similar to that of Lemma \ref{bchm3.6.12}.

We show the opposite inclusion as in the proof of Lemma \ref{bchm3.6.12}. 
Then it is sufficient to prove that $\phi_*B$ is $\mathbb{R}$-Cartier. 
However it is clear. 
In fact, $\phi_*(B\mid_{U_X})\simeq B|_{U_X}$ is $\mathbb{R}$-Cartier 
and $\phi_*B$ is also $\mathbb{R}$-Cartier in a neighborhood of $Y\setminus U_Y$, 
where $U_X$ and $U_Y$ are given by definition of special. 
\end{proof}

\begin{thm}\label{bchm3.11.2}
Let $\pi:X\to U$ be a projective morphism from a $\mathbb{Q}$-factorial surface 
to a quasi-projective variety. 
Let $V$ be a finite dimensional affine subspace of $\WDiv_\mathbb{R}(X)$, 
which is defined over the rationals. 
Fix a general ample $\mathbb{Q}$-divisor $A$ over $U$. 
Let $\phi:X\to Y$ be a birational contraction over $U$.

Then $\mathcal{W}_{\phi, A, \pi}(V)$ is a rational polytope. 
Moreover there are finitely many contraction morphisms $f_i:Y\to Z_i$ over $U$ $(1\leq i\leq m)$
such that if $f:Y\to Z$ is a contraction morphism over $U$ and there is an ample $\mathbb{R}$-divisor 
$D$ on $Z$ over $U$ with $\phi_*(K_X+\Delta)\sim_{\mathbb{R}, U}f^*D$, 
for some $\Delta\in\mathcal{W}_{\phi, A, \pi}(V)$, 
then there are an index $1\leq i\leq m$ and an isomorphism $\eta: Z_i\to Z$ such that $f=\eta\circ f_i$. 
\end{thm}

\begin{proof}
We take $\Delta\in\mathcal{W}_{\phi, A, \pi}(V)$ 
and put $\Gamma=\phi_*\Delta=\phi_*(A+B)$ and $C=\phi_*A$. 
Let $W$ be the affine subspace of $\WDiv_\mathbb{R}(Y)$ 
given by pushing forward the element of $V$ by $\phi$. 
We show that $\mathcal{N}_{C, \pi^{\prime}}(W)$ is a rational polytope, 
where $\pi^{\prime}:Y\to U$ is a structure morphism. 
Since $A$ is ample over $U$, we may assume that $\Delta\in[0, 1)$. 
Then by Lemma \ref{q-fac1} $Y$ is $\mathbb{Q}$-factorial. 
Therefore, by Lemma \ref{polytope1} $\mathcal{N}_{C, \pi^{\prime}}(W)$ is a rational polytope. 
Since $\Delta\in\mathcal{W}_{\phi, A, \pi}(V)$ if and only if 
$\Gamma=\phi_*\Delta\in\mathcal{N}_{C, \pi^{\prime}}(W)$, $K_X+\Delta-f^*f_*(K_X+\Delta)\geq0$ 
and $\Delta\in[0, 1]$, 
the first statement is clear.

Let $f:Y\to Z$ be a contraction morphism such that 
\[K_Y+\Gamma=K_Y+\phi_*\Delta\sim_{\mathbb{R}, U}f^*D\]
where $\Delta\in\mathcal{W}_{\phi, A, \pi}(V)$ and $D$ is 
an ample $\mathbb{R}$-divisor  over $U$ on $Z$. 
Then there is the unique face $G$ of $\mathcal{N}_{C, \pi^{\prime}}(W)$ 
such that $\Gamma$ belongs to the interior of $G$ 
and the curves contracted by $f$ are determined by $G$. 
Then there is the unique face $F$ of $\mathcal{W}_{\phi, A, \pi}(V)$ 
such that $\Delta$ belongs to the interior of $F$ 
and $G$ is determined by $F$. 
Since $\mathcal{W}_{\phi, A, \pi}(V)$ is a rational polytope, 
it has only finitely many faces $F$. 
We note that for any two contraction morphisms $f:Y\to Z$ and $f':Y\to Z'$ over $U$ 
there is an isomorphism $\eta:Z\to Z'$ with $f'=\eta\circ f$ 
if and only if the curves contracted by $f$ and $f'$ coincide. 
Thus the last statement is clear. 
\end{proof}

The following assertion is the log canonical version of Lemma \ref{bchm3.11.2}. 

\begin{thm}\label{bchm3.11.2+}
Let $\pi:X\to U$ be a projective morphism from a normal surface 
to a quasi-projective variety. 
Let $V$ be a finite dimensional affine subspace of $\WDiv_\mathbb{R}(X)$, 
which is defined over the rationals. 
Fix a general ample $\mathbb{Q}$-divisor $A$ over $U$. 
Let $\phi:X\to Y$ be a special birational contraction over $U$.

Then $\mathcal{W}_{\phi, A, \pi}(V)$ is a rational polytope. 
Moreover there are finitely many contraction morphisms $f_i:Y\to Z_i$ over $U$ $(1\leq i\leq m)$
such that if $f:Y\to Z$ is a contraction morphism over $U$ 
and there is an ample $\mathbb{R}$-divisor 
$D$ on $Z$ over $U$ with $\phi_*(K_X+\Delta)\sim_{\mathbb{R}, U}f^*D$, 
for some $\Delta\in\mathcal{W}_{\phi, A, \pi}(V)$, 
then there are an index $1\leq i\leq m$ and an isomorphism $\eta: Z_i\to Z$ such that $f=\eta\circ f_i$. 
\end{thm}

\begin{proof}
We use Lemma \ref{polytope2} instead of Lemma \ref{polytope1}. 
Then the proof of this statement is completely the same as the latter part of that of Lemma \ref{bchm3.11.2}. 
\end{proof}

\begin{thm}[Finiteness of models for $\mathbb{Q}$-factorial log surfaces]\label{finiteness-of-models}
Let $\pi:X\to U$ be a projective morphism from a $\mathbb{Q}$-factorial surface 
to a normal quasi-projective variety. 
Let $V$ be a finite dimensional affine subspace of $\WDiv_\mathbb{R}(X)$, 
which is defined over the rationals. 
Fix a general ample $\mathbb{Q}$-divisor $A$ over $U$.

Then there are finitely many birational contractions $\psi_i:X\to Z_i$ over $U$ $(1\leq i\leq n)$ 
such that if $\psi:X\to Z$ is a weak log canonical model of $K_X+\Delta$ 
over $U$ for some $\Delta\in\mathcal{B}_A(V)$, 
then there is an index $1\leq i\leq n$ and an isomorphism $\xi:Z_i\to Z$ such that $\psi=\xi\circ\psi_i$. 
\end{thm}

\begin{proof}
We take general ample $\mathbb{Q}$-divisors $H_1, \ldots, H_p$ over $U$, 
which generate $\WDiv_\mathbb{R}(X)$ modulo numerical equivalence over $U$. 
Replacing $V$ by the affine subspace of $\WDiv_\mathbb{R}(X)$ spanned by $V$ 
and the divisors $H_1, \ldots, H_p$, we may assume that 
$\mathcal{B}_A(V)$ spans $\WDiv_\mathbb{R}(X)$ modulo numerical equivalence over $U$. 
We apply \cite[Lemma 3.1]{hashizume} to $\mathcal{C}=\mathcal{B}_A(V)$. 
Then there are finitely many birational contractions $\phi_i:X\to Y_i$ over $U$ ($1\leq i\leq l$) 
such that if $\Delta'\in\mathcal{E}_{A, \pi}(V)$, 
then there is an index $1\leq i\leq l$ satisfying that $\phi_i$ 
is a weak log canonical model of $K_X+\Delta'$ over $U$. 
Moreover, by Theorem \ref{bchm3.11.2} for each index $1\leq i\leq m$ 
there are finitely many contraction morphism $f_{i, j}:Y_i\to Z_{i, j}$ over $U$ ($1\leq j\leq m$)
such that if $\Delta'\in\mathcal{W}_{\phi, A, \pi}(V)$ 
and there is a contraction morphism $f:Y_i\to Z$ over $U$ with 
\[K_{Y_i}+\Gamma_i=K_{Y_i}+{\phi_i}_*\Delta'\sim_{\mathbb{R}, U}f^*D\]
for some ample $\mathbb{R}$-divisor $D$ on $Z$ over $U$, 
then there are an index $1\leq j\leq m$ and an isomorphism $\xi:Z_{i, j}\to Z$ 
such that $f=\xi\circ f_{i, j}$. 
Then $\psi_k=f_{i, j}\circ \phi_i:X\to Z_{i, j}$ ($1\leq k\leq n$) are finitely many morphisms. 
We take a weak log canonical model $\psi:X\to Z$ of $K_X+\Delta$ over $U$ for $\Delta\in\mathcal{B}_A(V)$. 
Then we may assume that $\Delta$ is contained in the interior of $\mathcal{B}_A(V)$. 
In fact, since $\mathcal{B}_A(V)$ is a rational polytope, 
$\mathcal{B}_A(V)$ spans an affine subspace of $V_A$, which is defined over the rationals. 
Possibly replacing $V$, we may assume that $\mathcal{B}_A(V)$ spans $V_A$. 
We take a general $\mathbb{Q}$-divisor $\Delta_0\in[0, 1)$. 
By compactness, we can take $\Delta_1, \ldots, \Delta_q\in V_A$ 
such that $\mathcal{B}_A(V)$ is contained in the simplex spanned by $\Delta_1, \ldots, \Delta_q$. 
We fix $\epsilon>0$ sufficiently small such that 
\[\epsilon(\Delta_i-\Delta_0)+(1-2\epsilon)A\]
is an ample $\mathbb{Q}$-divisor over $U$ for $1\leq i\leq q$. 
We take general members $A_i\in|\epsilon(\Delta_i-\Delta_0)+(1-2\epsilon)A/U|_\mathbb{Q}$ 
and put $A'=\epsilon A$. 
Then we define $L:V_A\to L(V_A)=V'_{A'}$ by 
\[L(\Delta_i)=(1-\epsilon)\Delta_i+A_i+\epsilon\Delta_0
+A'-(1-\epsilon)A\sim_{\mathbb{Q}, U}\Delta_i,\]
and extend to the whole of $V_A$ linearly. 
Then $L$ is a $\mathbb{Q}$-linear isomorphism, 
which preserves $\mathbb{Q}$-linearly equivalence over $U$. 
Now $L$ is decomposed into $L_2\circ L_1$ such that 
\begin{align*}
L_1(\Delta_i)=\Delta_i+A_i/(1-\epsilon)+A'-A, 
\intertext{ and} 
L_2(\Delta)=(1-\epsilon)\Delta+\epsilon(A'+\Delta_0). 
\end{align*}
If $\Delta\in\mathcal{B}_A(V)$, then $L_1(\Delta)\in\mathcal{B}_{A'}(V')$. 
Therefore, $L(\Delta)$ is contained in the interior of $\mathcal{B}_{A'}(V')$. 
Then by Lemma \ref{q-fac1} $Z$ is $\mathbb{Q}$-factorial 
and so by Lemma \ref{bchm3.6.12} we may find $\Delta'\in\mathcal{E}_{A, \pi}(V)$ 
such that $\psi$ is the ample model of $K_X+\Delta'$ over $U$. 
We take an index $1\leq i\leq l$ such that $\phi_i$ is a weak log canonical model of $K_X+\Delta'$ over $U$. 
By (2) of Lemma \ref{ample-model} there is a contraction morphism $f:Y_i\to Z$ such that
\[K_{Y_i}+\Gamma'=f^*(K_Z+\Theta')\]
where $\Gamma'={\phi_i}_*\Delta'$ and $\Theta'=\psi_*\Delta'$. 
Since $\psi$ is the ample model of $K_X+\Delta'$ over $U$, 
$K_Z+\Theta'=\psi_*(K_X+\Delta')$ is ample over $U$. 
Thus there are an index $1\leq j\leq m$ and an isomorphism $\xi:Z_{i, j}\to Z$ 
such that $f=\xi\circ f_{i, j}$. 
Then 
\[\psi=f\circ\phi_i=\xi\circ f_{i, j}\circ\phi_i=\xi\circ\psi_j\]
for some index $1\leq j\leq n$. 
\end{proof}

The following assertion is the log canonical version of Lemma \ref{finiteness-of-models}. 

\begin{thm}[Finiteness of models for log canonical surfaces]\label{finiteness-of-models+}
Let $\pi:X\to U$ be a projective morphism from a normal surface 
to a normal quasi-projective variety. 
Let $V$ be a finite dimensional affine subspace of $\WDiv_\mathbb{R}(X)$, 
which is defined over the rationals. 
Fix a general ample $\mathbb{Q}$-divisor $A$ over $U$.

Then there are finitely many special birational contractions $\psi_i:X\to Z_i$ over $U$ $(1\leq i\leq n)$ 
such that if $\psi:X\to Z$ is a special weak log canonical model of $K_X+\Delta$ 
over $U$ for some $\Delta\in\mathcal{L}_A(V)$, 
then there is an index $1\leq i\leq n$ and an isomorphism $\xi:Z_i\to Z$ such that $\psi=\xi\circ\psi_i$. 
\end{thm}

\begin{proof}
First, we note that we can apply \cite[Lemma 3.1]{hashizume} to $\mathcal{C}=\mathcal{L}_A(V)$ in this setting with some modifications to the proof. 
In other words, we have finitely many special birational contractions 
$\phi_i:X\to Y_i$ over $U$ ($1\leq i\leq l$) 
such that if $\Delta'\in\mathcal{E}_{A, \pi}(V)$, 
then there is an index $1\leq i\leq l$ such that $\phi_i$ 
is a special weak log canonical model of $K_X+\Delta'$ over $U$. 
Then the rest of the proof is similar to that of Theorem \ref{finiteness-of-models}. 
\end{proof}

\begin{cor}\label{partition of wlcm}
Let $\pi:X\to U$ be a projective morphism from a $\mathbb{Q}$-factorial surface 
to a normal quasi-projective variety. 
Let $V$ be a finite dimensional affine subspace of $\WDiv_\mathbb{R}(X)$, 
which is defined over the rationals. 
Let $A$ be a general ample $\mathbb{Q}$-divisor over $U$.

Then there are finitely many birational contractions $\phi_i:X\to Y_i$ over $U$ $(1\leq i\leq l)$
such that
\[\mathcal{E}_{A, \pi}(V)=\displaystyle\bigcup^l_{i=1}\mathcal{W}_i\]
where each $\mathcal{W}_i=\mathcal{W}_{\phi_i, A, \pi}(V)$ is a rational polytope. 
Moreover if $\phi:X\to Y$ is a minimal model of $K_X+\Delta$ 
over $U$ for some $\Delta\in\mathcal{E}_{A, \pi}(V)$, 
then there is an index $1\leq i\leq l$ such that $\phi=\phi_i$. 
\end{cor}

\begin{proof}
By Theorem \ref{bchm3.11.2} and Theorem \ref{finiteness-of-models}, 
it is sufficient to prove that 
if $\Delta\in\mathcal{E}_{A, \pi}(V)$, 
then $K_X+\Delta$ has a minimal model over $U$. 
However, it is clear by Theorem \ref{mmp}. 
\end{proof}

The following assertion is the log canonical version of Lemma \ref{partition of wlcm}. 

\begin{cor}\label{partition of wlcm+}
Let $\pi:X\to U$ be a projective morphism from a normal surface 
to a normal quasi-projective variety. 
Let $V$ be a finite dimensional affine subspace of $\WDiv_\mathbb{R}(X)$, 
which is defined over the rationals. 
Let $A$ be a general ample $\mathbb{Q}$-divisor over $U$.

Then there are finitely many special birational contractions $\phi_i:X\to Y_i$ over $U$ $(1\leq i\leq l)$
such that
\[\mathcal{E}_{A, \pi}(V)=\displaystyle\bigcup^l_{i=1}\mathcal{W}_i\]
where each $\mathcal{W}_i=\mathcal{W}_{\phi_i, A, \pi}(V)$ is a rational polytope. 
Moreover if $\phi:X\to Y$ is a minimal model of $K_X+\Delta$ 
over $U$ for some $\Delta\in\mathcal{E}_{A, \pi}(V)$, 
then there is an index $1\leq i\leq l$ such that $\phi=\phi_i$. 
\end{cor}

\begin{proof}
This is an easy consequence of Theorem \ref{bchm3.11.2+} and Theorem \ref{finiteness-of-models+}. 
We note that by Lemma \ref{special2} a minimal model of $K_X+\Delta$ over $U$ is special.
\end{proof}


\section{The wall-crossing of ample models}\label{m-sec5} 

In this section, we will closely follow \cite[Section 3]{hacon-mackernan}. 

\begin{thm}{\em(cf.~\cite[Theorem 3.6]{hacon-mackernan})}\label{hm3.3}
Let $(Z, \Theta)$ be a projective $\mathbb{Q}$-factorial log surface. 
Let $V$ be a finite dimensional affine subspace of $\WDiv_\mathbb{R}(Z)$, 
which is defined over the rationals. 
Let $A\geq0$ be an ample $\mathbb{Q}$-divisor on $Z$. 
Then there are finitely many contraction morphisms 
$f_i:Z\to X_i$ $(1\leq i\leq l)$ with the following properties:
\begin{itemize}
\item[(1)] $\{ \mathcal{A}_i=\mathcal{A}_{A, f_i}(V)\mid 1\leq i\leq l\}$ 
is a partition of  $\mathcal{E}_{A}(V)$. 
$\mathcal{A}_i$ is a finite union of relative interior of rational polytopes. 
If $f_i$ is birational, then $\mathcal{C}_i$ is a rational polytope, 
where $\mathcal{C}_i$ is the closure of $\mathcal{A}_i$. 
\item[(2)] If $\mathcal{A}_j\cap\mathcal{C}_i\neq\emptyset$ for two indices $1\leq i, j\leq m$, 
then there is a contraction morphism $f_{i, j}:X_i\to X_j$ such that $f_j=f_{i, j}\circ f_i$. 
\end{itemize}

In addition, suppose that $V$ spans $\NS(Z)$, 
where $\NS(Z)$ is the N$\acute{e}$ron-Severi group of $Z$. 
\begin{itemize}
\item[(3)] We take $1\leq i\leq l$ such that a connected component $\mathcal{C}$ of $\mathcal{C}_i$ 
intersects the interior of $\mathcal{B}_A(V)$. 
Then the following are equivalent:
\begin{itemize}
\item[$\bullet$] $\mathcal{C}$ spnans $V$.
\item[$\bullet$] $f_i$ is birational and $X_i$ is $\mathbb{Q}$-factorial. 
\end{itemize}
\item[(4)] If $\mathcal{C}_i$ spans $V$ and $\Theta$ is a general point of $\mathcal{A}_j\cap\mathcal{C}_i$
which is also a point of the interior of $\mathcal{B}_A(V)$ for two indices $1\leq i, j\leq m$,
then $\mathcal{C}_i$ and $\NE(X_i/X_j)^*\times \mathbb{R}^k$ are locally isomorphic 
in a neighbourhood of $\Theta$ for some $k\geq0$. 
Moreover $\rho(X_i/X_j)=\dim\mathcal{C}_i-\dim\mathcal{C}_j\cap\mathcal{C}_i$.
\end{itemize}
\end{thm}

\begin{proof}
(1) The first statement is an easy consequence of Theorem \ref{bchm3.11.2}, 
Corollary \ref{partition of wlcm} and 
the uniqueness of ample models (cf.~Lemma \ref{ample-model}). 
Considering the construction of ample models, the last statement is clear (cf.~Proof of Theorem \ref{bchm3.11.2}).

(2) We take $\Theta\in\mathcal{A}_j\cap\mathcal{C}_i$ and $\Theta'\in\mathcal{A}_i$ 
such that 
\[\Theta_t=\Theta+t(\Theta'-\Theta)\in\mathcal{A}_i\]
for any $t\in(0, 1]$. 
By Corollary \ref{partition of wlcm}, there are $\delta>0$ and a birational contraction $f:Z\to X$ 
such that $f$ is a weak log canonical model of $K_Z+\Theta_t$ for any $t\in[0, \delta]$. 
Replacing $\Theta'=\Theta_1$ by $\Theta_\delta$, we may assume that $\delta=1$. 
We put $\Delta_t=f_*\Theta_t$. 
By the abundance theorem, $K_X+\Delta_t$ is semiample. 
Thus there is the induced contraction morphism $g_i:X\to X_i$ 
together with ample divisors $H_{1/2}$ and $H_1$
such that 
\[K_X+\Delta_{1/2}=g_i^*H_{1/2} \text{ and } K_X+\Delta_1=g_i^*H_1.\]
If we put $H_t=(2t-1)H_1+2(1-t)H_{1/2}$, then 
\begin{align*}
K_X+\Delta_t&=K_X+f_*(\Theta+t(\Theta'-\Theta))\\
&=(2t-1)(K_X+\Delta_1)+2(1-t)(K_X+\Delta_{1/2})\\
&=g_i^*H_t 
\end{align*}
for all $t\in[0, 1]$. 
As $K_X+\Delta_0$ is semiample and $g_i$ is a contraction morphism, $H_0$ is also semiample. 
Therefore, there is a contraction morphism $f_{i, j}:X_i\to X_j$ and this is the required.

In the rest of the proof, in addition, suppose that $V$ spans $\NS(Z)$.

(3) Suppose that $\mathcal{C}$ spans $V$. 
We take $\Theta$ in the interior of $\mathcal{C}\cap\mathcal{A}_i$. 
Let $f:Z\to X$ be a minimal model of $K_Z+\Theta$. 
By Lemma \ref{bchm3.6.12}, we obtain that $\mathcal{W}_{f, A}(V)=\mathcal{C}_{f, A}(V)$ 
and so there is an index $1\leq j\leq l$ such that $f=f_j$ and $\Theta\in\mathcal{C}_j$. 
Then $\mathcal{A}_j$ intersects with $\mathcal{A}_j$ 
and so we obtain $i=j$. 
Thus $f_i$ is birational and $X_i$ is $\mathbb{Q}$-factorial.

Conversely, suppose that $f_i$ is birational and $X_i$ is $\mathbb{Q}$-factorial. 
Fix $\Theta\in\mathcal{A}_i$. 
We take any divisor $B\in V$ 
such that $K_{X_i}+{f_i}_*(\Theta+B)$ is ample, 
$f_i$ is $(K_Z+\Theta+B)$-non-positive and $\Theta+B\in\mathcal{B}_A(V)$. 
As $f_i$ is birational, $\Theta+B\in\mathcal{A}_i$. 
Then $\mathcal{C}_i$ spans $V$.

(4) If we put $f=f_i$ and $X=X_i$, 
then $f$ is a $\mathbb{Q}$-factorial weak log canonical model of $K_Z+\Theta$ by (3). 
Let $E_1, \cdots, E_k$ be all $f$-exceptional divisors. 
We take divisors $B_i\in V$ numerically equivalent to $E_i$. 
If we put $E_0=\sum E_i$ and $B_0=\sum B_i$, then $E_0$ and $B_0$ are numerically equivalent. 
Since $\Theta$ belongs to the interior of $\mathcal{B}_A(V)$, 
there is a positive constant $\delta>0$ 
such that $\Theta+\delta E_0$ and $\Theta+\delta B_0$ belongs to the interior of $\mathcal{B}_A(V)$. 
Then $f$ is $(K_Z+\Theta+\delta E_0)$-negative 
and so $f$ is a minimal model of $K_Z+\Theta+\delta E_0$ 
and $f_j$ is the ample model of $K_Z+\Theta+\delta E_0$ corresponding to $f$. 
Thus $f$ is also a minimal model of $K_Z+\Theta+\delta B_0$ 
and $f_j$ is also the ample model of $K_Z+\Theta+\delta B_0$ (cf.~\cite[Lemma 3.6.9]{bchm}). 
In particular, $\Theta+\delta B_0\in\mathcal{A}_j\cap\mathcal{C}_i$. 
As $\Theta$ is general in $\mathcal{A}_j\cap\mathcal{C}_i$, 
$f$ is a minimal model of $K_Z+\Theta$. 
In particular, $f$ is $(K_Z+\Theta)$-negative.

We take $\epsilon>0$ such that if $\Xi\in V$ with $||\Xi-\Theta||<\epsilon$, 
then $\Xi$ belongs to the interior of $\mathcal{B}_A(V)$ and $f$ is $(K_Z+\Xi)$-negative. 
Then, $\Xi\in\mathcal{C}_i$ if and only if $K_X+\Delta=f_*(K_Z+\Xi)$ is nef. 
For any $(a_1, \ldots, a_k)\in\mathbb{R}^k$, we put $E=\sum a_iE_i$ and $B=\sum a_iB_i$. 
Let $W$ be the affine subspace of $\WDiv_\mathbb{R}(X)$ 
given by pushing forward the element of $V$ by $f$. 
As $\Xi+B$ is numerically equivalent to $\Xi+E$, 
if $||B||<\epsilon$, then $K_X+\Delta\in\mathcal{N}_{f_*A}(W)$ if and only if $K_X+\Delta+f_*B\in\mathcal{N}_{f_*A}(W)$. 
This means that $\mathcal{C}_i$ is locally isomorphic 
to $\mathcal{N}_{f_*A}(W)\times\mathbb{R}^k$.

However, since $f_j$ is the ample model of $K_Z+\Theta$, 
there is $\epsilon>0$ sufficiently small such that $K_X+\Delta$ is nef 
if and only if $K_X+\Delta$ is nef over $X_j$ (cf.~\cite[Corollary 3.11.3]{bchm}). 
There is a surjective affine linear map from $W$ to $\WDiv_\mathbb{R}(X)$ modulo numerical equivalence over $X_j$ 
and this induces an isomorphism 
\[\mathcal{N}_{f_*A}(W)\simeq\NE(X/X_j)^*\times\mathbb{R}^l\]
in a neighbourhood of $f_*\Delta$.

In the rest of the proof, we prove $l=0$ for (4). 
Suppose that $l\geq1$. 
We take a sufficiently small number $0<\gamma<1$. 
Then we may find a divisor $D\in V$ on $Z$ such that $K_X+f_*\Theta\pm\gamma f_*D$ 
is numerically trivial over $X_j$ and $f_*\Theta\pm\gamma f_*D\in \mathcal{B}_{f_*A}(V)$. 
By the abundance theorem, $K_X+f_*\Theta\pm\gamma f_*D$ is $\mathbb{R}$-linearly trivial over $X_j$. 
Thus $f_j$ is the ample model of $K_Z+\Theta\pm\gamma D$. 
However, this contradicts the fact that $\mathcal{C}_i$ is locally isomorphic 
to $\mathcal{N}_{f_*A}(W)\times\mathbb{R}^k$. 
\end{proof}

The following assertion is the log canonical version of Theorem \ref{hm3.3}. 

\begin{thm}\label{hm3.3+}
Let $Z$ be a normal projective surface. 
Let $V$ be a finite dimensional affine subspace of $\WDiv_\mathbb{R}(Z)$, 
which is defined over the rationals. 
Let $A\geq0$ be an ample $\mathbb{Q}$-divisor on $Z$. 
Then there are finitely many contraction morphisms $f_i:Z\to X_i$ $(1\leq i\leq l)$ with the following properties:
\begin{itemize}
\item[(1)] and {\em(2)} are the same statement as in Theorem \ref{hm3.3}. 
\end{itemize}

In addition, suppose that $V$ spans $\NS(Z)$, 
\begin{itemize}
\item[(3)] We take $1\leq i\leq l$ such that a connected component $\mathcal{C}$ of $\mathcal{C}_i$ 
intersects the interior of $\mathcal{L}_A(V)$. 
Then the following are equivalent:
\begin{itemize}
\item[$\bullet$] $\mathcal{C}$ spnans $V$.
\item[$\bullet$] $f_i$ is special. 
\end{itemize}
\item[(4)] If $\mathcal{C}_i$ spans $V$ and $\Theta$ is a general point of $\mathcal{A}_j\cap\mathcal{C}_i$
which is also a point of the interior of $\mathcal{L}_A(V)$ for two indices $1\leq i, j\leq l$,
then $\mathcal{C}_i$ and $\NE(X_i/X_j)^*\times \mathbb{R}^k$ are locally isomorphic 
in a neighbourhood of $\Theta$ for some $k\geq0$. 
Moreover $\rho(X_i/X_j)=\dim\mathcal{C}_i-\dim\mathcal{C}_j\cap\mathcal{C}_i$.
\end{itemize}
\end{thm}
\begin{proof}
The proof of (1), (2) and (4) is similar to that of Theorem \ref{hm3.3}. 
If $f$ is a minimal model of $K_Z+\Theta$, 
then by Lemma \ref{special2} $f$ is special. 
Thus (3) follows from Lemma \ref{bchm3.6.12+}. 
\end{proof}

We recall the following Bertini-type statement for the reader's convenience. 

\begin{cor}\label{hm3.4}
Let $Z$, $V$ and $A$ 
be the same notation as in Theorem \ref{hm3.3} (resp.~Theorem \ref{hm3.3+}). 
If $V$ spans $\NS(Z)$, then there is a dense Zariski open subset $U$ of the Grassmannian $G(\alpha, V)$ 
of real vector subspaces of dimension $\alpha$ such that 
if $[W]\in U$ and it is defined over the rationals, 
then $W$ satisfies (1)-(4) of Theorem \ref{hm3.3} (resp.~Theorem \ref{hm3.3+}). 
\end{cor}
\begin{proof}
See \cite[Corollary 3.3]{hacon-mackernan}. 
\end{proof}

For the rest of this section, we assume that 
$V$ has two dimensional and satisfies (1)-(4) of Theorem \ref{hm3.3} (resp.~Theorem \ref{hm3.3+}).

The following lemma is a key observation of the wall-crossing of ample models. 

\begin{lem}{\em(cf.~\cite[Lemma 3.5]{hacon-mackernan})}\label{hm3.5}
Let $Z$, $V$ and $A$ be the same notation as in Theorem \ref{hm3.3}. 
Let $f:Z\to X$ and $g:Z\to Y$ be two contraction morphisms such that 
$\mathcal{C}_{A, f}$ has two dimensional 
and $\mathcal{O}=\mathcal{C}_{A, f}\cap\mathcal{C}_{A, g}$ has one dimensional. 
Assume that $\rho(X)\geq\rho(Y)$ and that $\mathcal{O}$ 
is not contained in the boundary of $\mathcal{B}_A(V)$. 
Let $\Theta$ be a point belonging to the relative interior of $\mathcal{O}$ 
and let $\Delta=f_*\Theta$.

Then there is a contraction morphism $\pi:X\to Y$ 
such that $g=\pi\circ f$, $\rho(X)=\rho(Y)+1$, 
$\pi$ is a $(K_X+\Delta)$-trivial morphism 
and either
\begin{itemize}
\item[(1)] $\pi$ is a divisorial contraction, $\mathcal{O}\neq\mathcal{C}_{A, g}$ 
and $\mathcal{O}$ is not contained in the boundary of $\mathcal{E}_A(V)$, or
\item[(2)] $\pi$ is a Mori fiber space and $\mathcal{O}=\mathcal{C}_{A, g}$ 
is contained in the boundary of $\mathcal{E}_A(V)$. 
\end{itemize}
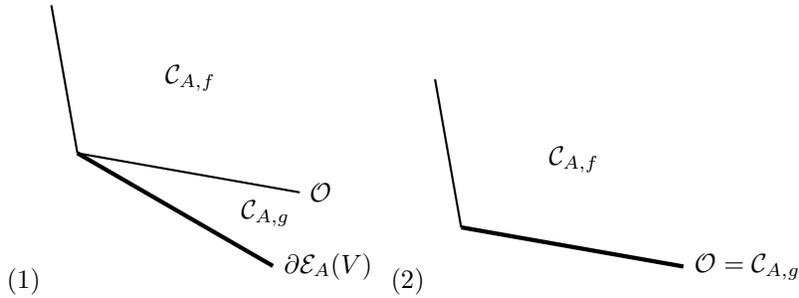
\begin{center}
{\em{(1)}}
\begin{tikzpicture}[thick]
\draw (0, 0) -- (100: 2); 
\draw (0, 0) -- (-10: 3)node[right]{$\mathcal{O}$}; 
\draw (0, 0) -- (-30: 3)[ultra thick]node[right]{$\partial\mathcal{E}_A(V)$}; 
\draw (3/2, 1) node{$\mathcal{C}_{A, f}$}; 
\draw (5/2, -0.8) node{$\mathcal{C}_{A, g}$}; 
\end{tikzpicture}
{\em{(2)}}
\begin{tikzpicture}[thick]
\draw (0, 0) -- (100: 2); 
\draw (0, 0) -- (-10: 3)[ultra thick]node[right]{$\mathcal{O}=\mathcal{C}_{A, g}$}; 
\draw (3/2, 0.9) node{$\mathcal{C}_{A, f}$}; 
\end{tikzpicture}
\end{center}
\end{lem}
\begin{proof}
By Theorem \ref{hm3.3}, $f$ is birational and $X$ is $\mathbb{Q}$-factorial. 
Let $h:Z\to W$ be the ample model of $K_Z+\Theta$. 
If $\Theta$ belongs to the boundary of $\mathcal{E}_A(V)$,  
then $K_Z+\Theta$ is not big. 
This is because by assumption $\Theta$ is not a point of the boundary of $\mathcal{B}_A(V)$. 
Then $h$ is not birational. 
If $\mathcal{O}$ is contained in the boundary of $\mathcal{E}_A(V)$, 
then $\mathcal{C}_{A, g}$ is one dimensional and so $\mathcal{O}=\mathcal{C}_{A, g}$. 
Thus by Theorem \ref{hm3.3} $\pi$ has a Mori fiber space structure. 
If $\mathcal{O}$ is not contained in the boundary of $\mathcal{E}_A(V)$ 
and $\mathcal{C}_{A, g}$ is one dimensional, 
then $\mathcal{O}=\mathcal{C}_{A, g}$ and so we obtain $K_Z+\Theta$ is big. 
Therefore, $\pi$ is a divisorial contraction 
and so $g$ is birational and $Y$ is $\mathbb{Q}$-factorial. 
Hence we obtain that $\mathcal{C}_{A, g}$ is two dimensional. 
However, this is a contradiction. 
Thus we may assume that $\mathcal{C}_{A, g}$ is two dimensional. 
Then we have $\mathcal{O}\neq\mathcal{C}_{A, g}$ 
and $\mathcal{O}$ is not contained in the boundary of $\mathcal{E}_A(V)$.

As $\mathcal{O}$ is a subset of both $\mathcal{C}_{A, f}$ and $\mathcal{C}_{A, g}$, 
there are two contraction morphisms $p:X\to W$ and $q:Y\to W$ 
of relative Picard number at most one by Theorem \ref{hm3.3}. 
Therefore, there are only two possibilities: 
\begin{itemize}
\item[(1)] $\rho(X)=\rho(Y)+1$, or
\item[(2)]  $\rho(X)=\rho(Y)$. 
\end{itemize}

We consider Case (1): $\rho(X)=\rho(Y)+1$. 
Then $h=g$ and $\pi=p:X\to Y$ is a contraction morphism such that $g=\pi\circ f$. 
As $g=h$ is birational, $\pi$ is a divisorial contraction.

We consider Case (2): $\rho(X)=\rho(Y)$. 
Then $\rho(X/W)=\rho(Y/W)=1$. 
Since $\mathcal{O}$ has one dimensional, $\mathcal{C}_{A, h}$ has one dimensional 
and so by (3) of Theorem \ref{hm3.3} $h$ is not birational or $W$ is not $\mathbb{Q}$-factorial. 
Thus $p$ and $q$ are not a divisorial contraction. 
Since $\mathcal{O}$ is not contained in the boundary of $\mathcal{E}_A(V)$, 
$p$ and $q$ are not Mori fiber spaces. 
Therefore, this case does not occur. 
\end{proof}

The following assertion is the log canonical version of Lemma \ref{hm3.5}. 

\begin{lem}\label{hm3.5+}
Let $(Z, \Theta)$, $V$ and $A$ be the same notation as in Theorem \ref{hm3.3+}. 
Let $f:Z\to X$ and $g:Z\to Y$ be two contraction morphisms such that 
$\mathcal{C}_{A, f}$ has two dimensional 
and $\mathcal{O}=\mathcal{C}_{A, f}\cap\mathcal{C}_{A, g}$ has one dimensional. 
Assume that $\rho(X)\geq\rho(Y)$ and that $\mathcal{O}$ 
is not contained in the boundary of $\mathcal{L}_A(V)$. 
Let $\Theta$ be a point belonging to the relative interior of $\mathcal{O}$ 
and let $\Delta=f_*\Theta$.

Then there is a contraction morphism $\pi:X\to Y$ 
such that $g=\pi\circ f$, $\rho(X)=\rho(Y)+1$, $\pi$ is a $(K_X+\Delta)$-trivial morphism 
and either
\begin{itemize}
\item[(1)] $\pi$ is a divisorial contraction, $\mathcal{O}\neq\mathcal{C}_{A, g}$ 
and $\mathcal{O}$ is not contained in the boundary of $\mathcal{E}_A(V)$, or
\item[(2)] $\pi$ is a Mori fiber space and $\mathcal{O}=\mathcal{C}_{A, g}$ 
is contained in the boundary of $\mathcal{E}_A(V)$. 
\end{itemize}
\end{lem}
\begin{proof}
We note that by Theorem \ref{special1} a divisorial contraction is special. 
Then we can use Theorem \ref{hm3.3+} instead of Theorem \ref{hm3.3}. 
Therefore, this statement is proved similarly to Lemma \ref{hm3.5}. 
\end{proof}

\begin{lem}\label{hm3.6}
Let $f:W\to X$ be a birational contraction of projective $\mathbb{Q}$-factorial surfaces. 
Let $(W, \Theta)$ and $(W, \Phi)$ be log surfaces. 
If $f$ is the ample model of $K_W+\Theta$ and $\Theta-\Phi$ is ample, 
then $f$ is the result of running the $(K_W+\Phi)$-minimal model program. 
\end{lem}
\begin{proof}
We take an ample divisor $H$ on $W$ 
such that $\Phi+H\in[0, 1]$, 
$K_W+\Phi+H$ is ample 
and $tH\sim_\mathbb{R}\Theta-\Phi$ 
for some positive real number $0<t<1$. 
Then $f$ is the ample model of $K_W+\Phi+tH$. 
We take any $s<t$ sufficiently close to $t$. 
Since $f$ is $(K_W+\Phi+tH)$-non-positive and $H$ is ample, $f$ is $(K_W+\Phi+sH)$-negative. 
Then $f$ is the unique minimal model of $K_W+\Phi+sH$. 
In particular, if we run the $(K_W+\Phi+sH)$-minimal model program with scaling of $H$ 
and the value of the scalar is $s$, 
then the induced morphism is $f$. 
\end{proof}

The following assertion is the log canonical version of Lemma \ref{hm3.6}. 

\begin{lem}\label{hm3.6+}
Let $f:W\to X$ be a special birational contraction of normal projective surfaces. 
Let $(W, \Theta)$ and $(W, \Phi)$ be log canonical. 
If $f$ is the ample model of $K_W+\Theta$ and $\Theta-\Phi$ is ample, 
then $f$ is the result of running the $(K_W+\Phi)$-minimal model program. 
\end{lem}
\begin{proof}
We have already seen that $f_*(K_W+\Phi+tH)$ 
is $\mathbb{R}$-Cartier in  the proof of Lemma \ref{bchm3.6.12+}. 
Thus this statement is proved similarly to Lemma \ref{hm3.6}. 
\end{proof}

We now adopt additional notation for the rest of this section. 
A point $\Theta=A+B$ is contained in the boundary of $\mathcal{E}_A(V)$ 
and the interior of $\mathcal{B}_A(V)$ (resp.~$\mathcal{L}_A(V)$). 
If $\Theta$ is contained in only one polytope, 
then we assume that it is a zero dimensional face of $\mathcal{E}_A(V)$ (see the following figure). 
Let $\mathcal{T}_1, \cdots, \mathcal{T}_k$ be all two dimensional 
rational polytopes $\mathcal{C}_i$ containing $\Theta$. 
Possibly re-ordering, we may assume that the intersections 
$\mathcal{O}_0$ and $\mathcal{O}_k$ of $\mathcal{T}_0$ and $\mathcal{T}_k$ 
with the boundary of $\mathcal{E}_A(V)$ 
and $\mathcal{O}_i=\mathcal{T}_i\cap\mathcal{T}_{i+1}$ $(1\leq i\leq k-1)$ 
and all $\mathcal{O}_i$ have one dimensional. 
Let $f_i:Z\to X_i$ be contraction morphism associated to $\mathcal{T}_i$ 
and let $g_i:Z\to S_i$ be contraction morphisms associated to $\mathcal{O}_i$. 
We put $f=f_1:Z\to X=X_1$, $g=f_k:Z\to Y=X_k$, $X'=X_2$ and $Y'=X_{k-1}$. 
Let $\phi:X\to S=S_0$ and $\psi:Y\to T=S_k$ be the induced morphism 
and let $h:Z\to R$ be the ample model $K_Z+\Theta$. 

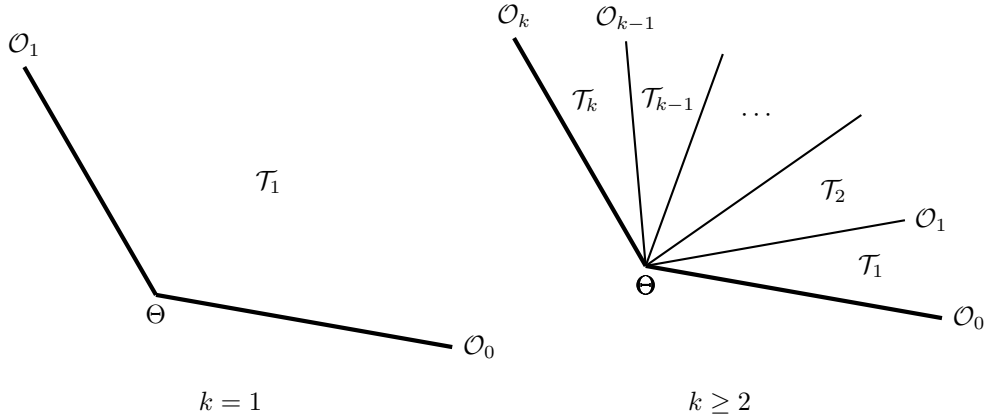
\begin{multicols}{2}
\begin{center}
\begin{tikzpicture}[thick]
\draw (0, 0)node[below]{$\Theta$} -- (120: 3.5)[ultra thick]node[above]{$\mathcal{O}_1$}; 
\draw (0, 0) -- (-10: 4)[ultra thick]node[right]{$\mathcal{O}_0$}; 
\draw (1.5, 1.5) node{$\mathcal{T}_1$}; 
\end{tikzpicture}
\[k=1\]

\begin{tikzpicture}[thick]
\draw (0, 0)node[below]{$\Theta$} -- (120: 3.5)[ultra thick]node[above]{$\mathcal{O}_k$}; 
\draw (0, 0)node[below]{$\Theta$} -- (95: 3)node[above]{$\mathcal{O}_{k-1}$}; 
\draw (0, 0)node[below]{$\Theta$} -- (70: 3); 
\draw (0, 0)node[below]{$\Theta$} -- (35: 3.5); 
\draw (0, 0)node[below]{$\Theta$} -- (10: 3.5)node[right]{$\mathcal{O}_1$}; 
\draw (0, 0) -- (-10: 4)[ultra thick]node[right]{$\mathcal{O}_0$}; 
\draw (3, 0) node{$\mathcal{T}_1$}; 
\draw (2.5, 1) node{$\mathcal{T}_2$}; 
\draw (1.5, 2) node{$\cdots$}; 
\draw (0.3, 2.2) node{$\mathcal{T}_{k-1}$}; 
\draw (-0.8, 2.2) node{$\mathcal{T}_k$}; 
\end{tikzpicture}
\[k\geq2\]
\end{center}
\end{multicols}

\begin{thm}\label{hm3.7}
Let $(Z, \Phi)$ be a $\mathbb{Q}$-factorial log surface. 
Suppose that $\Theta-\Phi$ is ample.

Then $\phi$ and $\psi$ are Mori fiber spaces 
which are outputs of the $(K_Z+\Phi)$-minimal model progaram. 
Moreover $\phi$ and $\psi$ are connected by Sarkisov links,
where each $f_i$ is the result of running the $(K_Z+\Phi)$-minimal model program. 
\end{thm}
\begin{proof}
By Lemma \ref{hm3.5}, there is the following commutative diagram: 

$$
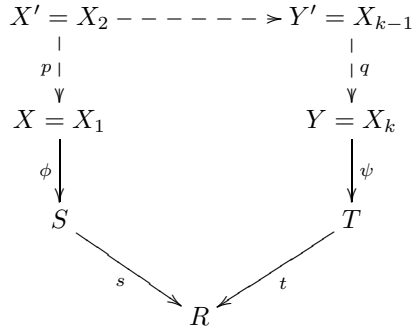
\xymatrix{
X' =X_2\ar@{-->}[d]_p\ar@{-->}[rr] & & Y'=X_{k-1} \ar@{-->}[d]^q\\
X=X_1 \ar[d]_\phi & & Y=X_k \ar[d]^\psi\\ 
S \ar[rd]_s & & T \ar[ld]^t\\
& R &
}
$$

\noindent where $p$ and $q$ are birational and $\phi$ and $\psi$ are Mori fiber spaces. 
We note that there is a morphism between $X_i$ and $X_{i+1}$. 
We can take $\Theta_i$ contained in $\mathcal{T}_i$ 
such that $\Theta_i-\Phi$ is ample. 
As $\mathcal{T}_i$ has two dimensional, $X_i$ is $\mathbb{Q}$-factorial 
and so $f_i:Z\to X_i$ is the result of running 
the $(K_Z+\Phi)$-minimal model program by Lemma \ref{hm3.6}. 
By Theorem \ref{hm3.3}, there is a contraction morphism $X_i\to R$ 
with $\rho(X_i/R)\leq2$. 
If $\rho(X_i/R)=1$, then $X_i\to R$ has a Mori fiber space structure. 
By Theorem \ref{hm3.3}, $\mathcal{C}_{A, h}$ is one dimensional. 
Thus there is a facet of $\mathcal{T}_i$ 
which is contained in the boundary of $\mathcal{E}_A(V)$ 
and so $i=1$ or $i=k$. 
Hence, if $k\geq4$, then $\rho(X_2)=\cdots=\rho(X_{k-1})$. 
However by Lemma \ref{hm3.5} this is impossible. 
If $\rho(X/R)=2$ or $\rho(Y/R)=2$, then $k=1$ or $k=2$. 
Therefore, there are only four possibilities: 
\begin{itemize}
\item[(1)] $k=1$, 
\item[(2)] $k=2$, $\rho(X/R)=1$ and $\rho(Y/R)=2$,
\item[(3)] $k=2$, $\rho(X/R)=2$ and $\rho(Y/R)=1$, or 
\item[(4)] $k=3$. 
\end{itemize}

We consider Case (1). 
Then $X=Y$. Since $\mathcal{O}_0\neq\mathcal{O}_1$, $\rho(X/R)=2$ and so $R$ is a point. 
Thus we obtain a link of type (IV).

We consider Case (2). 
Since $\rho(X/R)=1$ and $\rho(Y/R)=2$, $s$ is the identity and $R$ is a point and so we obtain a link of type (I).

We consider Case (3). 
This case is the same as Case (1) and so we obtain a link of type (III).

Finally, we consider Case (4). 
Since $\rho(X'/R)=2$ and $\rho(X/R)=1$, 
$s$ is the identity. 
Similarly $t$ is also the identity 
and so we obtain a link of type (II). 
\end{proof}

The following assertion is the log canonical version of Theorem \ref{hm3.7}. 

\begin{thm}\label{hm3.7+}
Let $(Z, \Phi)$ be a log canonical surface. 
Suppose that $\Theta-\Phi$ is ample.

Then $\phi$ and $\psi$ are Mori fiber spaces 
which are outputs of the $(K_Z+\Phi)$-minimal model progaram. 
Moreover $\phi$ and $\psi$ are connected by Sarkisov links, 
where each $f_i$ is the result of running the $(K_Z+\Phi)$-minimal model program. 
\end{thm}
\begin{proof}
It is sufficient 
to replace Theorem \ref{hm3.3}, Lemma \ref{hm3.5} and Lemma \ref{hm3.6} 
with Theorem \ref{hm3.3+}, Lemma \ref{hm3.5+} and Lemma \ref{hm3.6+}, respectively, 
in the proof of Theorem \ref{hm3.7}. 
\end{proof}


\section{Proof of Main Theorems}\label{m-sec6} 

\begin{lem}\label{hm4.1}
Let $(Z, \Phi)$ be a projective $\mathbb{Q}$-factorial log surface. 
Let $f:Z\to X$ and $g:Z\to Y$ be birational contractions. 
Let $\phi:X\to S$ and $\psi:Y\to T$ be two Mori fiber spaces 
which are outputs of the $(K_Z+\Phi)$-minimal model program.

Then we may find an ample $\mathbb{Q}$-divisor $A$ on $Z$ 
and a two dimensional affine subspace $V$ of $\WDiv_\mathbb{R}(Z)$, 
which is defined over the rationals, 
with the following properties: 
\begin{itemize}
\item[(1)] $\Theta-\Phi$ is ample for any $\Theta\in\mathcal{B}_A(V)$, 
\item[(2)] $\mathcal{A}_{A, \phi\circ f}$ and $\mathcal{A}_{A, \psi\circ g}$ 
are not contained in the boundary of $\mathcal{B}_A(V)$, 
\item[(3)] $V$ satisfies (1)-(4) of Theorem \ref{hm3.3}, 
\item[(4)] $\mathcal{C}_{A, f}$ and $\mathcal{C}_{A, g}$ have two dimensional, 
\item[(5)] $\mathcal{C}_{A, \phi\circ f}$ and $\mathcal{C}_{A, \psi\circ g}$ have one dimensional. 
\end{itemize}
\end{lem}
\begin{proof}
We put $\Delta=f_*\Phi$ and $\Gamma=g_*\Phi$. 
We take general very ample divisors $G_0, \ldots, G_k\geq 0$ 
such that $G_1, \ldots, G_k$ generate $\NS(Z)$ 
and put $A=G_0/4$, $H_i=G_i/2$ and $H=A+H_1+\cdots+H_k$. 
We take sufficiently ample divisors $C$ on $S$ and $D$ on $T$ 
such that 
\[-(K_X+\Delta)+\phi^*C \text{ and } -(K_Y+\Gamma)+\psi^*D\]
are both ample. 
If we take a sufficiently small rational number $0<\delta<1$, then 
\[-(K_X+\Delta+\delta f_*H)+\phi^*C \text{ and } -(K_Y+\Gamma+\delta g_*H)+\psi^*D\]
are also both ample and $f$ and $g$ are both $(K_Z+\Phi+\delta H)$-negative 
since ampleness and negativity of contractions are both open conditions. 
Replacing $H$ by $\delta H$, 
we may assume that $\delta=1$. 
We take a $\mathbb{Q}$-divisor $\Phi\geq\Phi_0\geq 0$ 
such that 
\begin{gather*}
A+(\Phi_0-\Phi), \\
-(K_X+f_*\Phi_0+f_*H)+\phi^*C \\
\intertext{ and } -(K_Y+g_*\Phi_0+g_*H)+\psi^*D
\end{gather*}
are all ample and $f$ and $g$ are both $(K_Z+\Phi_0+H)$-negative. 
We take general $\mathbb{Q}$-divisors $F_1\geq 0$ and $G_1\geq 0$ 
such that 
\begin{gather*}
F_1\sim_\mathbb{Q}-(K_X+f_*\Phi_0+f_*H)+\phi^*C, \\
G_1\sim_\mathbb{Q}-(K_Y+g_*\Phi_0+g_*H)+\psi^*D \\
\intertext{and}
\lfloor \Phi_0+H+F+G\rfloor=0 
\end{gather*}
where $F=f^*F_1$ and $G=g^*G_1$. 
Let $V_0$ be the affine subspace of $\WDiv_\mathbb{R}(Z)$, 
which is the translation by $\Phi_0$ 
of the affine subspace spanned by $H_1, \ldots, H_k, F$ and $G$. 
Suppose that $\Theta=A+B$ is contained in $\mathcal{B}_A(V_0)$. 
Then 
\[\Theta-\Phi=(A+\Phi_0-\Phi)+(B-\Phi_0)\]
is ample since $A+\Phi_0-\Phi$ is ample and $B-\Phi_0$ is nef by the definition of $V_0$. 
Since $f$ is $(K_Z+\Phi_0+H)$-negative, it is $(K_Z+\Phi_0+F+H)$-negative 
and $g$ is similarly $(K_Z+\Phi_0+G+H)$-negative. 
Thus $\Phi_0+F+H\in\mathcal{A}_{A, \phi\circ f}(V_0)$, 
$\Phi_0+G+H\in\mathcal{A}_{A, \psi\circ g}(V_0)$, 
$f$ is a minimal model of $K_Z+\Phi_0+F+H$ 
and $g$ is a minimal model of $K_Z+\Phi_0+G+H$. 
Then $V_0$ satisfies (1)-(4) of Theorem \ref{hm3.3}.

Since $H_1, \ldots, H_k$ generate $\NS(Z)$, 
there are some real constants $h_1, \ldots, h_k$ 
such that $G$ is numerically equivalent to $\sum h_iH_i$. 
Then $\Phi_0+F+\delta G+H-\delta(\sum h_iH_i)$ and $\Phi_0+F+H$
are numerically equivalent. 
If $\delta>0$ is sufficiently small, 
then $\Phi_0+F+\delta G+H-\delta(\sum h_iH_i)\in\mathcal{B}_A(V_0)$. 
Thus $\mathcal{A}_{A, \phi\circ f}$ and $\mathcal{A}_{A, \psi\circ g}$ 
are not contained in the boundary of $\mathcal{B}_A(V_0)$. 
Now by (3) of Theorem \ref{hm3.3} each of $\mathcal{A}_{A, f}(V_0)$ and $\mathcal{A}_{A, g}(V_0)$ spans $V_0$. 
Therefore, since $\rho(X/S)=\rho(Y/T)=1$, 
each of $\mathcal{A}_{A, \phi\circ f}(V_0)$ and $\mathcal{A}_{A, \psi\circ g}(V_0)$ spans 
the affine hyperplane of $V_0$.

Let $V_1$ be the translation by $\Phi_0$ of the two dimensional affine subspace 
spanned by $F+H-A$ and $G+H-A$. 
Let $V$ be a small general perturbation of $V_1$, 
which is defined over the rationals. 
Then (2) holds and
(1) always holds for any subspace of $V_0$. 
By Corollary \ref{hm3.4}, (3) holds 
and then we see that (4) and (5) hold. 
\end{proof}

The following assertion is the log canonical version of Lemma \ref{hm4.1}. 

\begin{lem}\label{hm4.1+}
Let $(Z, \Phi)$ be a projective log canonical surface. 
Let $f:Z\to X$ and $g:Z\to Y$ be birational contractions. 
Let $\phi:X\to S$ and $\psi:Y\to T$ be two Mori fiber spaces 
which are outputs of the $(K_Z+\Phi)$-minimal model program.

Then we may find an ample $\mathbb{Q}$-divisor $A$ on $Z$ 
and a two dimensional affine subspace $V$ of $\WDiv_\mathbb{R}(Z)$, 
which is defined over the rationals, 
with the following properties: 
\begin{itemize}
\item[(1)] $\Theta-\Phi$ is ample for any $\Theta\in\mathcal{L}_A(V)$, 
\item[(2)] $\mathcal{A}_{A, \phi\circ f}$ and $\mathcal{A}_{A, \psi\circ g}$ 
are not contained in the boundary of $\mathcal{L}_A(V)$, 
\item[(3)] $V$ satisfies (1)-(4) of Theorem \ref{hm3.3}, 
\item[(4)] $\mathcal{C}_{A, f}$ and $\mathcal{C}_{A, g}$ have two dimensional, 
\item[(5)] $\mathcal{C}_{A, \phi\circ f}$ and $\mathcal{C}_{A, \psi\circ g}$ have one dimensional. 
\end{itemize}
\end{lem}
\begin{proof}
This statement is proved as in Lemma \ref{hm4.1}. 
\end{proof}

\begin{proof}[Proof of Theorem \ref{main-theorem1}]
We take an ample $\mathbb{Q}$-divisor $A$ on $Z$ 
and a two dimensional affine subspace $V$ of $\WDiv_\mathbb{R}(Z)$ 
given by Lemma \ref{hm4.1}. 
We take $\Theta_0\in\mathcal{A}_{A, \phi\circ f}(V)$ and $\Theta_1\in\mathcal{A}_{A, \psi\circ g}(V)$ 
belonging to the interior of $\mathcal{B}_A(V)$. 
As $V$ has two dimensional, removing two points $\Theta_0$ and $\Theta_1$ 
divides the boundary of $\mathcal{E}_A(V)$ into the two parts. 
The part of the boundary consisting only of divisors 
which is not big is contained in the interior of $\mathcal{B}_A(V)$ 
by the definition of pseudo-effective divisors. 
If we trace this part from $\Theta_0$ to $\Theta_1$, 
then we obtain finitely many points $\Theta_i$ $(2\leq i\leq k)$ 
which  are contained in at most  three polytopes $\mathcal{C}_{A, h_{i,j}}(V)$. 
We note that if $\Theta_i$ is contained in only one polytope, 
then this polytope is a corner of $\mathcal{E}_A(V)$. 
Then Theorem \ref{hm3.7} implies that there is a Sarkisov link $\sigma:X_i\dashrightarrow Y_i$ 
and $\sigma$ is connected by these links. 

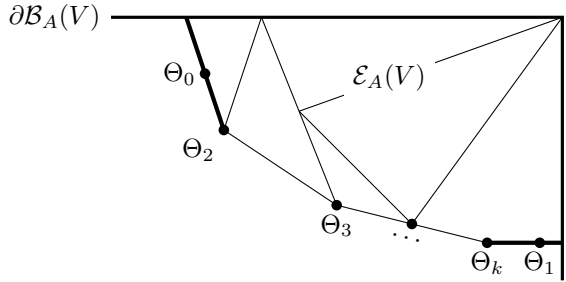
\begin{tikzpicture}
\draw[very thick]  (3, -0.5) -- (3, 3) -- (-3, 3)node[left] {$\partial\mathcal{B}_A(V)$} ; 
\draw (-1.5, 1.5) -- (0, 1/2) -- node[below, sloped] {$\cdots$} (2, 0);
\draw[ultra thick] (-2, 3) -- (-1.5, 1.5);
\draw[ultra thick] (2, 0) -- (3, 0); 
\draw (-1.5, 1.5) -- (-1, 3); 
\draw (0, 1/2) -- (-1, 3); 
\draw ($(0, 1/2)!1/2!(2, 0)$) -- ($(0, 1/2)!1/2!(-1, 3)$); 
\draw ($(0, 1/2)!1/2!(-1, 3)$) -- (3, 3); 
\draw ($(0, 1/2)!1/2!(2, 0)$) -- (3, 3); 


\node[fill=white] at (0.7, 2.2) {$\mathcal{E}_A(V)$}; 

\fill ($(-2, 3)!1/2!(-1.5, 1.5)$)node[left] {$\Theta_0$} circle (2pt); 
\fill (2.7,0)node[below] {$\Theta_1$} circle (2pt); 

\fill (-1.5, 1.5)node[below left] {$\Theta_2$} circle (2pt); 
\fill (0, 1/2)node[below] {$\Theta_3$} circle (2pt); 
\fill ($(0, 1/2)!1/2!(2, 0)$) circle (2pt); 
\fill (2, 0)node[below] {$\Theta_k$} circle (2pt); 
\end{tikzpicture}
\end{proof}

\begin{proof}[Proof of Theorem \ref{main-theorem2}]
This statement is also proved as in Theorem \ref{main-theorem1}. 
\end{proof}


\section{Appendix}

In this section, we prove that $T$ (resp.~$S$) in a link of Type (I) (resp.~(III)), 
$S$ and $T$ in a link of Type (IV) 
are isomorphic to $\mathbb{P}^1$. 

\begin{prop}\label{P^1-1}
In a link of Type (I), $T=\mathbb{P}^1$. 
\end{prop}
\begin{proof} We divide the proof into two cases, where the characteristic of the base field $k$ is zero or positive.

{\bfseries Case 1:} $\Char(k)=0$.

We have the following diagram: 
$$
\xymatrix{
X'\ar@{=}[r]\ar[d]_\alpha&Y\ar[d]^\psi \\
X\ar[d]_\phi&T\ar[ld] \\
\text{pt.}
}
$$
Then $(X, \Delta)$ is a $\mathbb{Q}$-factorial log surface or a log canonical surface 
such that $-(K_X+\Delta)$ is ample. 
If $(X, \Delta)$ is log canonical, then by the Kodaira vanishing theorem (cf.~\cite[Theorem 5.6.4]{fujino-book}) 
we obtain 
\[H^1(X, \mathcal{O}_X)=0.\]
Next suppose that $(X, \Delta)$ is a $\mathbb{Q}$-factorial log surface. 
Since ampleness is an open condition, we may assume that $\lfloor \Delta\rfloor=0$. 
Then $(X, \Delta)$ is kawamata log terminal outside finitely many points. 
Now we consider the exact sequence 
\[\cdots\to H^1(X, \mathcal{J}(X, \Delta))\to H^1(X, \mathcal{O}_X)\to 
H^1(X, \mathcal{O}_X/\mathcal{J}(X, \Delta))\to \cdots\]
where $\mathcal{J}(X, \Delta)$ is the multiplier ideal sheaf associated to the pair $(X, \Delta)$. 
By the Nadel vanishing theorem, $H^1(X, \mathcal{J}(X, \Delta))=0$. 
Since $\dim_k\Supp(\mathcal{O}_X/\mathcal{J}(X, \Delta))=0$, 
we obtain $H^1(X, \mathcal{O}_X/\mathcal{J}(X, \Delta))=0$. 
Thus we have 
\[H^1(X, \mathcal{O}_X)=0.\]

Now $\alpha:X'\to X$ is a divisorial contraction 
and we can take a boundary divisor $\Delta'$ 
such that $-(K_{X'}+\Delta')$ is $\alpha$-ample. 
Thus by the Kodaira vanishing theorem (cf.~\cite[Theorem 6.2]{fujino-tanaka}) we obtain 
\[R^1\alpha_*\mathcal{O}_{X'}=0.\]
By the Leray spectral sequence, 
$H^1(X', \mathcal{O}_{X'})\simeq H^1(X, \mathcal{O}_X)=0$. 
Moreover, applying the Leray spectral sequence to $\psi:Y\to T$, we obtain 
\[0\to H^1(T, \mathcal{O}_T)\to H^1(Y, \mathcal{O}_Y)\to \cdots.\]
Thus we obtain $H^1(T, \mathcal{O}_T)=0$ and so $T\simeq\mathbb{P}^1$

{\bfseries Case 2:} $\Char(k)>0$. 

If $X$ is $\mathbb{Q}$-factorial, then we may assume that $\lfloor \Delta\rfloor=0$. \\
Next suppose that $(X, \Delta)$ is log canonical 
such that $\lfloor \Delta\rfloor\neq0$ and $X$ is not $\mathbb{Q}$-factorial. 
We take a sufficiently ample divisor $H$ 
and we take a general member $B\sim_\mathbb{Q} \lfloor \Delta\rfloor+H$. 
Then 
\[K_X+\Delta+\epsilon H\sim_\mathbb{Q} K_X+\Delta-\epsilon\lfloor \Delta\rfloor+\epsilon B\]
for $0<\epsilon<1$. 
We put $\Delta_1=\Delta-\epsilon\lfloor \Delta\rfloor+\epsilon B$. 
If a rational number $\epsilon$ is sufficiently small, 
then $(X, \Delta_1)$ is a log surface such that $\lfloor \Delta_1\rfloor=0$ 
and $-(K_X+\Delta_1)$ is ample.
Thus, anyway,  we can apply Theorem \ref{oo} 
and there is no non-trivial torsion line bundle on $X$.

We can take a boundary divisor $\Delta'$ 
such that $-(K_{X'}+\Delta')$ is $\alpha$-ample as in Case 1. 
We take a torsion line bundle $\mathcal{L}'$ on $X'$. 
By the basepoint-free theorem (cf.~\cite[Theorem 13.12]{fujino-surfaces2}), there is a line bundle $\mathcal{L}$ 
such that $\mathcal{L}'\simeq\alpha^*\mathcal{L}$. 
Thus $\mathcal{L}$ is also a torsion line bundle. 
Since there is no non-trivial torsion line bundle on $X$, $\mathcal{L}=\mathcal{O}_X$ 
and so we obtain $\mathcal{L}'=\mathcal{O}_{X'}$. 
Therefore, there is also no non-trivial torsion line bundle on $X'=Y$.

We take a torsion line bundle $\mathcal{M}$ on $T$. 
Then $\phi^*\mathcal{M}$ is also a torsion line bundle on $Y$. 
Thus $\phi^*\mathcal{M}\simeq\mathcal{O}_Y$ and so we have $\mathcal{M}\simeq \mathcal{O}_T$. 
This means that there is no non-trivial torsion line bundle on $T$. 
Therefore $T\simeq\mathbb{P}^1$. 
\end{proof}

By symmetry, we obtain $S=\mathbb{P}^1$ in a link of Type (III). 

\begin{prop}\label{P^1-2}
In a link of Type (IV), $S=T=\mathbb{P}^1$. 
\end{prop}
\begin{proof}
Since $\rho(X)=2$, $\phi$ and $\psi$ contract different rational curves. 
We take a general fiber $F$ of $\phi$. 
Then $F$ is isomorphic to $\mathbb{P}^1$. 
Therefore $T\simeq\mathbb{P}^1$. 
By symmetry, we obtain $S\simeq\mathbb{P}^1$. 
\end{proof}


\end{document}